\documentclass{amsart}
\usepackage{amssymb}
\usepackage{amsmath, amscd}
\usepackage{amsthm}
\usepackage{mathrsfs}
\usepackage{perpage}
\usepackage[all,cmtip]{xy}
\usepackage{setspace}
\usepackage{amsbsy}
\usepackage[foot]{amsaddr}
\usepackage{url}
\usepackage[T1]{fontenc}
\usepackage{verbatim}
\usepackage{mathrsfs}
\usepackage{ifthen}
\usepackage{blkarray}
\usepackage{multirow}
\usepackage{enumitem}
\usepackage{chngcntr}
\usepackage{hyperref}
\usepackage{etoolbox}
\usepackage{tikz}
\usetikzlibrary{matrix,arrows,decorations.pathmorphing}
\usepackage[normalem]{ulem}
\usepackage[left=2cm,head=13pt,
top=1.5cm,right=2cm, bottom=1.5cm]{geometry}

\setenumerate[0]{label=(\alph*)}

\numberwithin{equation}{subsection}

\AtBeginEnvironment{equation}{\setcounter{equation}{\value{subsubsection}}}{}{}{}
\AtEndEnvironment{equation}{\stepcounter{subsubsection}}{}{}{}

\newtheorem{theorem}[subsubsection]{Theorem}
\newtheorem{lemma}[subsubsection]{Lemma}

\newtheorem{example}[subsubsection]{Example}
\newtheorem{corollary}[subsubsection]{Corollary}
\newtheorem{definition}[subsubsection]{Definition}
\newtheorem{question}[subsubsection]{Question}

\newtheorem{proposition}[subsubsection]{Proposition}

\newtheorem{lemmaapp}[subsection]{Lemma}
\newtheorem{corapp}[subsection]{Corollary}

\theoremstyle{remark}

\newtheorem{rmk}[subsubsection]{Remark}

\newtheorem{rmkapp}[subsection]{Remark}
\newtheorem{construction}[subsubsection]{Construction}

\title[quasi-constant-classification]{Quasi-constant characters:
\\
Motivation, classification and applications }

\numberwithin{equation}{subsection}

\newcommand{\GZip}{\mathop{\text{$G$-{\tt Zip}}}\nolimits}
\newcommand{\gzipmu}{\GZip^{\mu}}

\newskip\procskipamount
\newskip\interskipamount
\newskip\refskipamount
\newcommand{\procskip}{\vskip\procskipamount}
\newcommand{\interskip}{\vskip\interskipamount}
\newcommand{\refskip}{\vskip\refskipamount}

\newcommand{\procbreak}{\par
   \ifdim\lastskip<\procskipamount\removelastskip
   \penalty-100
   \procskip\fi
   \noindent\ignorespaces}

\newcommand{\titlebreak}{\par%
\ifdim\lastskip<\interskipamount\removelastskip%
\penalty10000%
\interskip\fi%
\noindent}%

\newcommand{\interbreak}{\par%
\ifdim\lastskip<\interskipamount\removelastskip%
\penalty-100%
\interskip\fi%
\noindent\ignorespaces}%

\newcommand{\refbreak}{\par%
\ifdim\lastskip<\refskipamount\removelastskip%
\penalty-100%
\refskip\fi%
\noindent\ignorespaces}%

\newcounter{listcounter}
\newcounter{deflistcounter}
\newcounter{equivcounter}

\newskip{\itemsepamount}
\newskip{\topsepamount}
\newenvironment{assertionlist}{%
  \begin{list}
    {\upshape (\arabic{listcounter})}
    {\setlength{\leftmargin}{18pt}
     \setlength{\rightmargin}{0pt}
     \setlength{\itemindent}{0pt}
     \setlength{\labelsep}{5pt}
     \setlength{\labelwidth}{13pt}
     \setlength{\listparindent}{\parindent}
     \setlength{\parsep}{0pt}
     \setlength{\itemsep}{\itemsepamount}
     \setlength{\topsep}{\topsepamount}
     \usecounter{listcounter}}}
  {\end{list}}

\newenvironment{definitionlist}{%
  \begin{list}
    {\upshape (\alph{deflistcounter})}
    {\setlength{\leftmargin}{18pt}
     \setlength{\rightmargin}{0pt}
     \setlength{\itemindent}{0pt}
     \setlength{\labelsep}{5pt}
     \setlength{\labelwidth}{13pt}
     \setlength{\listparindent}{\parindent}
     \setlength{\parsep}{0pt}
     \setlength{\itemsep}{\itemsepamount}
     \setlength{\topsep}{\topsepamount}
     \usecounter{deflistcounter}}}
  {\end{list}}

\newenvironment{equivlist}{%
  \begin{list}
    {\upshape (\roman{equivcounter})}
    {\setlength{\leftmargin}{18pt}
     \setlength{\rightmargin}{0pt}
     \setlength{\itemindent}{0pt}
     \setlength{\labelsep}{5pt}
     \setlength{\labelwidth}{13pt}
     \setlength{\listparindent}{\parindent}
     \setlength{\parsep}{0pt}
     \setlength{\itemsep}{\itemsepamount}
     \setlength{\topsep}{\topsepamount}
     \usecounter{equivcounter}}}
  {\end{list}}

% ...

\newcommand{\Kcal}{{\mathcal K}}

\newcommand{\Pcal}{{\mathcal P}}

\newcommand{\mfr}{{\mathfrak m}}

\newcommand{\rfr}{{\mathfrak r}}

\newcommand{\Dfr}{{\mathfrak D}}

\newcommand{\CC}{\mathbf{C}}

\newcommand{\FF}{\mathbf{F}}
\newcommand{\GG}{\mathbf{G}}
\newcommand{\HH}{\mathbf{H}}

\newcommand{\LL}{\mathbf{L}}

\newcommand{\PP}{\mathbf{P}}
\newcommand{\QQ}{\mathbf{Q}}
\newcommand{\RR}{\mathbf{R}}
\renewcommand{\SS}{\mathbf{S}}
\newcommand{\TT}{\mathbf{T}}

\newcommand{\XX}{\mathbf{X}}
\newcommand{\YY}{\mathbf{Y}}
\newcommand{\ZZ}{\mathbf{Z}}

\DeclareMathOperator{\Pic}{Pic}

\newcommand{\Lscr}{{\mathscr L}}

\newcommand{\Sscr}{{\mathscr S}}

\newcommand{\Vscr}{{\mathscr V}}

\newcommand{\cent}{{\rm Cent}}
\newcommand{\gm}{\GG_{\textnormal{m}}}

\newcommand{\qbar}{\overline{\mathbf Q}}

\newcommand{\fp}{\mathbf F_p}

\newcommand{\fpbar}{\overline{\FF}_p}
\newcommand{\galq}{{\rm Gal}(\qbar / \QQ)}
\newcommand{\galfp}{{\rm Gal}(\fpbar / \fp)}

\newcommand{\chargp}{X^*}
\newcommand{\cochargp}{X_*}

\newcommand{\chargpt}{\chargp(T)}

\newcommand{\cochargpt}{\cochargp(T)}

\newcommand{\zgeqo}{\mathbf Z_{\geq 1}}
\newcommand{\zgeqz}{\mathbf Z_{\geq 0}}
\newcommand{\qgeqz}{\QQ_{\geq 0}}

\DeclareMathOperator{\ad}{ad}

\DeclareMathOperator{\an}{an}

\DeclareMathOperator{\der}{der}

\DeclareMathOperator{\GS}{GS}

\DeclareMathOperator{\Lie}{Lie}

\DeclareMathOperator{\pr}{pr}

\DeclareMathOperator{\res}{Res}

\DeclareMathOperator{\Sh}{Sh}

\DeclareMathOperator{\stab}{Stab}
\DeclareMathOperator{\std}{Std}

\newcommand{\kbar}{\bar k}
\newcommand{\gal}{{\rm Gal}}

\newcommand{\galk}{\gal(\overline{k}/k)}

\newcommand{\gofr}{\mathbf G(\mathbf R)}

\newcommand{\gofaf}{\mathbf G(\mathbf A_f)}

\newcommand{\af}{\mathbf A_f}

\newcommand{\Th}{{\rm Th.}}
\newcommand{\Ths}{{\rm Ths.}}
\newcommand{\Rmk}{{\rm Rmk.}}

\newcommand{\Cor}{{\rm Cor.}}

\newcommand{\Chap}{{\rm Chap.}}

\newcommand{\Def}{{\rm Def.}}
\newcommand{\Defs}{{\rm Defs.}}
\newcommand{\Prop}{{\rm Prop.}}

\newcommand{\loccit}{{\em loc.\ cit. }}
\newcommand{\loccitn}{{\em loc.\ cit.}}

\newcommand{\opcitn}{{\em op.\ cit.}}

\newcommand{\Cf}{{\em Cf. }}
\newcommand{\cf}{{\em cf. }}

\newcommand{\ie}{i.e.,\ }
\newcommand{\eg}{e.g.,\ }

\newcommand{\diag}{{\rm diag}}

\newcommand{\fil}{{\rm Fil}}

\newcommand{\gx}{(\GG, \XX)}
\newcommand{\shgx}{\Sh\gx}
\newcommand{\shgxk}{\shgx_{\mathcal K}}

\newcommand{\gsgx}{\GS\gx}
\newcommand{\gsgxk}{\gsgx_{\Kcal}}

\newcommand{\llambda}{\Lscr(\lambda)}

\newcommand{\dR}{{\rm dR}}

\newcommand{\xg}{\mathbf X_{g}}

\newcommand{\shdagsp}{(GSp(2g), \xg)}

\email{wushijig@gmail.com}
\address{W. G. Department of Mathematics, Stockholm University, Stockholm SE-10691, Sweden}

\email{jeanstefan.koskivirta@gmail.com}

\address{J.-S. K. Department of Mathematics, South Kensington Campus,
Imperial College London,
London
SW7 2AZ, UK }

\date{\today}
\author{Wushi Goldring and Jean-Stefan Koskivirta}
\begin{document}
\begin{abstract} In \cite{Goldring-Koskivirta-Strata-Hasse}, initially motivated by questions about the Hodge line bundle of a Hodge-type Shimura variety, we singled out a generalization of the notion of {\em minuscule character} which we termed {\em quasi-constant}. Here we prove that the character of the Hodge line bundle is always quasi-constant. Furthermore, we classify the quasi-constant characters of an arbitrary connected, reductive group over an arbitrary field. As an application,  we observe that, if $\mu$ is a quasi-constant cocharacter of an $\fp$-group $G$, then  our construction of group-theoretical Hasse invariants in \loccit applies to the stack $\gzipmu$, without any restrictions on $p$, even if the pair $(G, \mu)$ is not of Hodge type and even if $\mu$ is not minuscule. We conclude with a more speculative discussion of some further motivation for considering quasi-constant cocharacters in the setting of our program outlined in \loccit

\end{abstract}

\pagestyle{plain}
\maketitle
\tableofcontents
\newpage
\section{Introduction}
\label{sec-intro}
This paper is the fourth installment in a series on our program to connect the three areas (A) {\em Automorphic Algebraicity}, (B) {\em $G$-Zip-Geometricity} and (C) {\em Griffiths-Schmid Algebraicity}.
Our program was introduced in \cite{Goldring-Koskivirta-Strata-Hasse} and developed further in \cite{Goldring-Koskivirta-zip-flags,Goldring-Koskivirta-Diamond-I}. For more advances in the program, see our forthcoming joint work with Stroh and Brunebarbe \cite{Brunebarbe-Goldring-Koskivirta-Stroh-ampleness}. Some key aspects of the program are discussed in \S\ref{sec-test-case} below.

The present paper dissects the notion of `quasi-constant character' introduced in \cite[\Def~N.5.3]{Goldring-Koskivirta-Strata-Hasse}. The idea behind quasi-constancy is to isolate those (co)characters which are simplest from the point of view of pairings with Weyl-Galois orbits of (co)roots.  The quasi-constant condition simultaneously incorporates those of `minuscule' and `cominuscule'. As observed in \loccitn, it is also well-adapted to the study of (i) the Hodge line bundle of a symplectic embedding of Shimura
varieties, (ii) the existence of group-theoretical Hasse invariants on stacks $\gzipmu$. 

The following recalls the definition of quasi-constant characters and proceeds to summarize the topics covered in the main body of the text.

\subsection{Quasi-constant characters}
\label{sec-intro-quasi-constant}
Throughout this article, fix a field $k$ and a connected, reductive $k$-group $G$.
Let $T$ be a maximal torus in $G$ (defined over $k$). Write \begin{equation}
\label{eq-root-datum-G}
(X^*(T), \Phi; X_*(T), \Phi^{\vee})\end{equation} for the root datum of the pair $(G_{\bar k}, T_{\bar k})$, where:
$X^*(T)$ (resp. $X_*(T)$) denotes the character (resp. cocharacter) group of $T_{\bar k}$ and  $\Phi=\Phi(G,T)$ (resp. $\Phi^{\vee}=\Phi^{\vee}(G,T)$) denotes the set of roots (resp. coroots) of $T_{\kbar}$ in $G_{\kbar}$. Denote the perfect pairing $X^*(T) \times X_*(T) \to \ZZ$ by $\langle, \rangle$. Set $W$ to be the Weyl group of $T_{\kbar}$ in $G_{\kbar}$.

In \cite{Goldring-Koskivirta-Strata-Hasse}, our investigation of Hasse invariants on Ekedahl-Oort strata of Hodge-type Shimura varieties led us to single out the following notion (see \Def~N.5.3 of \loccitn):

\begin{definition} A character $\chi \in X^*(T)$ is \uline{quasi-constant} if, for every root $\alpha \in \Phi$ satisfying $\langle \chi , \alpha^{\vee} \rangle \neq 0$ and all $\sigma \in W \rtimes \galk $, one has $$  \frac{\langle \chi, \sigma \alpha^{\vee} \rangle}{\langle \chi, \alpha^{\vee} \rangle } \in \{-1,0,1\}. $$
\label{def quasi constant} \end{definition}

Note the resemblance with the definition of `minuscule character' (recalled in \S\ref{sec-minuscule}). One defines quasi-constant cocharacters in the same way, by replacing coroots with roots. It is then clear that, over an algebraically closed field, a quasi-constant character (resp. cocharacter) of a pair $(G, T)$ is the same thing as a quasi-constant cocharacter (resp. character) of the dual pair $(G^{\vee}, T^{\vee})$ associated to the dual of the root datum~\eqref{eq-root-datum-G}.
\subsection{Classification} 
\label{sec-intro-classification}
One of the main results of this paper is a classification of quasi-constant (co)characters for any connected reductive group $G$, over any field $k$.  
The classification is given in two steps: \Th~\ref{th-quasi-constant-cominuscule} treats the case that $k$ is algebraically closed and $G$ is simple and simply-connected (resp. adjoint). \Th~\ref{th-gen-quasi-constant-cominuscule} explains how the general classification reduces to the former special case. Note that in this paper, `semisimple and simply-connected' is always taken in the sense of root data, \ie it means that the $\ZZ$-span of $\Phi^{\vee}$ is $X_*(T)$.

The classification of quasi-constant (co)characters is in terms of minuscule and cominuscule (co)characters. For the convenience of the reader, the latter two notions are recalled in \S\S\ref{sec-minuscule}--\ref{sec-cominuscule}.

\begin{theorem}
\label{th-quasi-constant-cominuscule}
Suppose $k$ is an algebraically closed field and $G$ is a simple and simply-connected (resp. adjoint) $k$-group. A character (resp. cocharacter) of $T$ is quasi-constant if and only if it is a multiple of one which is either minuscule or cominuscule.
\end{theorem}

\begin{rmk} \label{rmk-cominuscule} It is clear from the definitions that a multiple of a minuscule (co)character is quasi-constant. Moreover, if $\chi \in \chargpt$ is cominuscule but not minuscule, then by looking at tables (\cf Bourbaki \cite[\Chap~VI, Planches I-IX]{bourbaki-lie-4-6} or Knapp \cite[Appendix~C.1-C.2]{Knapp-beyond-intro-book}) one finds that $G$ is of type $B_n$ or $C_n$ ($n \geq 2$); in type $C_n$ the character $\chi$ is (conjugate to) the fundamental weight corresponding to the unique long simple root, while in type $B_n$ it is (conjugate to) the fundamental weight corresponding to the extremal vertex of the Dynkin diagram which is farthest from the unique short simple root. In both of these cases, one checks that $\chi$ is quasi-constant. Thus, the primary content of
\Th~\ref{th-quasi-constant-cominuscule} is that there are no other characters which are quasi-constant. \end{rmk}
\begin{rmk} \label{part 1 implies part 2 of thm} In the same vein as \Rmk~\ref{rmk-cominuscule}, when the root system $\Phi$ is simply-laced, the quasi-constant characters of $T$ are exactly the multiples of the minuscule ones.
\end{rmk}

Let $\tilde{G}$ be the simply-connected cover of the derived subgroup of $G$. Let $G^{\ad}$ denote the adjoint quotient of $G$.

\begin{theorem}
\label{th-gen-quasi-constant-cominuscule}

Suppose $k$ is an arbitrary field and $G$ is a connected, reductive $k$-group.
 
\begin{enumerate}
\item \label{item-red-k-simple}
A character (resp. cocharacter) of $T$ is quasi-constant if and only if its pullback to every $k$-simple factor of $\tilde G$ is quasi-constant (resp. its projection to every $k$-simple factor of $G^{\ad}$ is quasi-constant). 
\item 
\label{item-k-simple}
Suppose $G$ is $k$-simple and simply-connected (resp. adjoint).  
A character (resp. cocharacter) of $T$ is quasi-constant if and only if, with respect to the $\bar k$-simple factors of $\tilde G_{\bar k}$ (resp. $G_{\bar k}^{\ad}$), it has the form $m(\xi_1, \ldots , \xi_d)$, where $m, \xi_1, \ldots , \xi_n$ satisfy  
\begin{enumerate}[label=(\roman*)]
\item $m \in \zgeqo$;
\item Every $\xi_i$ is either trivial, minuscule or cominuscule;
\item The nontrivial $\xi_j$ are either all minuscule or all cominuscule.
\end{enumerate}
\end{enumerate}
 \end{theorem}
\begin{rmk} A particularly easy case of \Th~\ref{th-gen-quasi-constant-cominuscule} is the following: Assume $G$ is absolutely simple. Then $\chi \in X^*(T)$ (resp. $\mu \in X_*(T)$) is quasi-constant if and only if its pullback to $\tilde G$ (resp. projection onto $G^{\ad}$) is.

\end{rmk}
\subsection{Duality}
\label{sec-intro-duality} 
When $G$ is semisimple, there is a duality between the rays spanned by quasi-constant cocharacters and quasi-constant characters. For general reductive $G$, this duality still allows to associate a quasi-constant character to a quasi-constant cocharacter (and vice-versa), albeit in a non-canonical way. 

A ray in a $\QQ$-vector space will mean the $\qgeqz$ multiples of a nonzero vector, \ie a one-dimensional cone.   
\begin{definition}
\label{def-quasi-constant lines} A ray $\rfr$ in $X^*(T)_{\QQ}$  (resp. $X_*(T)_{\QQ}$) is called quasi-constant if some (equivalently every) element of $X^*(T) \cap \rfr$ (resp. $X_*(T) \cap \rfr$) is quasi-constant.\footnote{Throughout, a subscript `$\QQ$' indicates base change from $\ZZ$ to $\QQ$.}
\end{definition}
\begin{proposition}[see Construction~\ref{const-duality} and \Prop~\ref{prop-duality}] \label{prop-intro-duality} Suppose $G$ is semisimple. Given a choice of simple roots $\Delta \subset \Phi$, the linear map $X_*(T) \to X^*(T)$ which associates to a fundamental coweight the corresponding fundamental weight \textnormal{(\S\ref{sec-fund-weights})} restricts to a bijection $\rfr \leftrightarrow \rfr^{\vee}$ between $\Delta$-dominant, quasi-constant rays in $X_*(T)$ and those in $X^*(T)$. 
This bijection satisfies the following properties:
\begin{enumerate}
\item 
\label{item-intro-duality-restriction-levi}
The quasi-constant ray $\rfr^{\vee}$ is the restriction of a ray in $X^*(\cent(\rfr))_{\QQ}$ (see \textnormal{\Rmk~\ref{rmk-cent-line}}).
\item 
\label{item-intro-duality-maximal}
The Levi $\cent(\rfr)$ of $G$ is the maximal Levi satisfying~\ref{item-intro-duality-restriction-levi}.

\end{enumerate}

 \end{proposition}

\subsection{Applications I: Shimura varieties of Hodge type}
\label{sec-intro-shimura}
Consider a symplectic embedding 
\begin{equation}
\label{eq-symplectic-embedding}
\psi: \gx \hookrightarrow   \shdagsp \end{equation}
 of a Shimura datum of Hodge type $\gx$  into a Siegel Shimura datum $\shdagsp$. Given a neat, open, compact subgroup $\Kcal \subset \gofaf$, let $\shgx_{\Kcal}$ denote the associated Shimura variety at level $\Kcal$ over $\CC$. There exists $\Kcal_g \subset GSp(2g, \af)$ such that $\psi(\Kcal) \subset \Kcal_g$ and $\psi$ induces a closed embedding of $\shgx_{\Kcal}$ into $\Sh\shdagsp_{\Kcal_g}$ (\cf \cite[1.15]{Deligne-Travaux-Shimura}).
The Hodge line bundle $\omega_g$ of the Siegel Shimura variety $\Sh\shdagsp_{\Kcal_g}$ is defined as \begin{equation}
\omega_g:=\det \fil^1H^1_{\dR},
\end{equation} where $H^1_{\dR}$ is the universal weight one variation of Hodge structure over $\Sh\shdagsp_{\Kcal_g}$ and $\fil^1$ refers to the Hodge filtration. For $g \geq 2$, sections of $\omega_g^k$ are what are most classically called "Siegel modular forms of weight $k$ and level $\Kcal_g$". 
The Hodge line bundle $\omega(\psi)=\omega(\gx, \psi)$ of the pair $(\gx, \psi)$ on the Shimura variety $\shgx_{\Kcal}$ is then defined by pullback:
\begin{equation}
\omega(\psi):=\psi^*(\omega_g).
\end{equation}

Choose $h \in \XX$. Define $h_g \in \XX_g$ by $h_g:=\psi \circ h$. As usual, set $\mu=(h\otimes \CC)(z,1)$ and $\mu_g=(h_g \otimes \CC)(z,1)$. One has $\mu \in X_*(\GG)$ and $\mu_g \in X_*(GSp(2g))$. Let $E$ be the reflex field of $\gx$. The centralizers $\LL:=\cent_{\GG_E}(\mu)$ and $\LL_g:=\cent_{GSp(2g)_{E}}(\mu_g)$ 
are Levi subgroups of $\GG_E$ and $GSp(2g)_E$. The line bundle $\omega_g$ arises from a character $\eta_g$ of $\LL_g$; the line bundle $\omega(\psi)$ arises from the character $\psi^*\eta_g$ of $\LL$. The character $\eta_{\omega}(\psi):=\psi^*\eta_g$ is called the \underline{Hodge character} of the symplectic embedding $\psi$.   
\begin{theorem} 
\label{th-intro-omega-quasi-constant}
For every symplectic embedding~\eqref{eq-symplectic-embedding}, the Hodge character $\eta_{\omega}(\psi)=\psi^*\eta_g$ is quasi-constant.
\end{theorem}
\Th~\ref{th-intro-omega-quasi-constant} was applied in \cite{Goldring-Koskivirta-Strata-Hasse} to show that Ekedahl-Oort strata of Hodge-type Shimura varieties admit Hasse invariants at all primes $p \neq 2$ of good reduction (see \S4.3 of \loccitn). For further applications of quasi-constant characters to Hasse invariants, see \S\ref{sec-intro-hasse}. The proof of \Th~\ref{th-intro-omega-quasi-constant} given in \S\ref{sec-hodge-q-constant} was previously given in Appendix A of an earlier draft of \loccit

Let $\tilde{\LL}$ be the preimage of $\LL$ in $\tilde{\GG}_E$ and write $s: \tilde \LL \to \LL$ for the natural map (in particular $\tilde \LL$ is not the simply-connected cover of the derived group of $\LL$). 
The following invariance property of the Hodge character and Hodge line bundle under functoriality is a simple consequence of \Th~\ref{th-hasse-quasi-constant}:
\begin{corollary} 
\label{cor-intro-hodge-line-invariant} Assume that the adjoint group $\GG^{\ad}$ is $\QQ$-simple.  Then the positive ray generated by the (pullback of) the Hodge character $s^*\psi^*\eta_g$ in $X^*(\tilde{\LL})_{\QQ}$ is independent of the choice of embedding $\psi$. In other words, the positive ray generated by the Hodge line bundle  $\omega(\psi)$ in the Picard group $\Pic(\shgxk)_{\QQ}$ is independent of $\psi$.\footnote{Note that the Picard group here is the usual one of line bundles without additional structure; greater care must be taken if one wants a statement concerning line bundles which are equivariant with respect to a group action, \eg the $\gofaf$-action related to the action of Hecke algebras on spaces of automorphic forms.}
\end{corollary}
\begin{rmk} \label{rmk-hodge-line} It is easy to give examples of two embeddings $\psi_1, \psi_2$ such that $\omega(\psi_2)$ is a nontrivial positive multiple of $\omega(\psi_1)$ (\cf
 \cite[\S2.1.6, Footnote 7]{Goldring-Koskivirta-Strata-Hasse}). Furthermore, we explain in \S\ref{sec-ex-G-not-simple} below why the assumption that $\GG^{\ad}$ is $\QQ$-simple is essential. Thus 
\Cor~\ref{cor-intro-hodge-line-invariant} exhibits the best possible invariance property of the Hodge line bundle under functoriality.
\end{rmk}
\subsection{Applications II: Group-theoretical Hasse invariants}
\label{sec-intro-hasse}
In this \S, suppose $p$ is a prime and $k=\fp$. So $G$ is an $\fp$-group. Let $\mu \in X_*(G)$. Pink-Wedhorn-Ziegler associate to the pair $(G, \mu)$ a zip datum and a stack $\GZip^{\mu}$ of $G$-Zips of type $\mu$ \cite{PinkWedhornZiegler-F-Zips-additional-structure,Pink-Wedhorn-Ziegler-zip-data}. The stack $\GZip^\mu$ admits a stratification parameterized by a certain subset ${}^I W$ of the Weyl group $W$. The zip stratification of $\GZip^\mu$ is a group-theoretic generalization of the Ekedahl-Oort stratification. \Cf \cite{Goldring-Koskivirta-Strata-Hasse,Goldring-Koskivirta-zip-flags} for the basic facts about $\GZip^\mu$, including the connection with the special fibers of Hodge-type Shimura varieties.

Let $w \in {}^I W$, $S_w$ the corresponding zip stratum and $\overline{S}_w$ its Zariski closure. Let $L:=\cent(\mu)$ and $\lambda \in X^*(L)$. There is an associated line bundle $\Vscr(\lambda)$ on $\GZip^\mu$. Recall from the introduction of \cite{Goldring-Koskivirta-zip-flags} that a \uline{group-theoretical Hasse invariant} or \uline{characteristic section} for  $(\lambda,  S_w)$ is a section $t \in H^0(\overline{S}_w, \Vscr(n\lambda))$ for some $n\geq 1$,  whose non-vanishing locus is precisely $S_w$. Recall further that the stratification of $\gzipmu$ is termed \uline{principally pure} if every stratum admits a characteristic section for some $\lambda \in X^*(L)$ and \uline{uniformly principally pure} if a single $\lambda$ admits characteristic sections on all strata. In the latter case, such a $\lambda$ is called a \uline{Hasse generator} for $\GZip^\mu$. 

One of the basic questions studied in \cite{Goldring-Koskivirta-Strata-Hasse} and \cite{Goldring-Koskivirta-zip-flags} was:
\begin{question} \label{q-principally-pure} For what pairs $(G, \mu)$ is the zip stratification of $\gzipmu$  (uniformly) principally pure?
\end{question}

In \cite[\Th~3.2.3]{Goldring-Koskivirta-Strata-Hasse}, it was shown that $\GZip^{\mu}$ is uniformly principally pure as long as $p$ satisfies a mild bound in terms of $(G, \mu)$ (see \S\ref{sec-purity} for the precise result). An explicit bound is recorded in Appendix~\ref{app-bounds}. As an application of the quasi-constancy of the Hodge line bundle (\Th~\ref{th-intro-omega-quasi-constant}), it was shown that, when $(G, \mu)$ arises from a Shimura datum of Hodge-type, the zip stratification is uniformly principally pure (without any assumption on $p$). These results were reproved in \cite{Goldring-Koskivirta-zip-flags} by a somewhat different method, using zip data of higher exponent.
Finally, a counter-example to principal purity when $p=2$ was given in \cite[\S5.3]{Goldring-Koskivirta-zip-flags}.

In this paper, the classification and duality of quasi-constant characters are used to improve upon the results of \cite{Goldring-Koskivirta-Strata-Hasse} and \cite{Goldring-Koskivirta-zip-flags}. 
\begin{theorem} \label{th-hasse-quasi-constant} Suppose $G$ is an $\fp$-group and $\mu \in X_*(G)$ is a quasi-constant cocharacter. Then
\begin{enumerate}
\item \label{item-levi-admits-quasi-constant} \textnormal{Construction~\ref{const-duality}} equips the Levi $L:=\cent(\mu)$ of $G$ with a quasi-constant character $\mu^*$.
\item \label{item-quasi-constant-implies-hasse-inv-generator} The quasi-constant character $-\mu^*$ afforded by part~\ref{item-levi-admits-quasi-constant} is a Hasse generator for $\gzipmu$. Consequently, the stratification of $\gzipmu$ is uniformly principally pure.
\end{enumerate}
\end{theorem}
 We stress that \Th~\ref{th-hasse-quasi-constant} contains no assumption on $p$ and makes no reference to Shimura varieties.
In particular, it provides a result for all $p$ in some cases when $\mu$ is not minuscule. An interesting feature of \Th~\ref{th-hasse-quasi-constant} is that it uses both quasi-constant characters and cocharacters simultaneously.

\subsection{Outline}
\label{sec-outline} \S\ref{sec-notation} sets up the basic notation and structure theory concerning reductive groups that is used in the rest of the paper. \S\ref{sec-classification} concerns the classification and duality of quasi-constant (co)characters. The classification (\Ths~\ref{th-quasi-constant-cominuscule} and~\ref{th-gen-quasi-constant-cominuscule}) is proved in \S\S\ref{sec-proof-simple}-\ref{sec-proof-generalization}; the duality construction~\ref{const-duality} is given in \S\ref{sec-duality}. The quasi-constancy of the Hodge line bundle (\Th~\ref{th-intro-omega-quasi-constant}) is established in \S\ref{sec-hodge-q-constant}. 

\S\ref{sec-further-motivation} discusses further applications, motivation and open questions concerning quasi-constant (co)characters. \S\ref{sec-purity} gives the application to uniform principal purity (\Th~\ref{th-hasse-quasi-constant}). Motivation for the quasi-constant condition as a unification of `minuscule' and `cominuscule' is provided in \S\ref{sec-motivation-unification-co-min}. Finally \S\ref{sec-test-case} includes a more speculative discussion of the potential role of the quasi-constant condition in our program: We mention open questions concerning Griffiths-Schmid manifolds and stacks of $G$-Zips and how quasi-constant cocharacters offer an interesting test case for these questions.    

Appendix~\ref{app-bounds} records explicit bounds for the uniform principal purity of $\GZip^\mu$ depending only on the type of $G$ and that of $L$.

\section*{Acknowledgements}
This paper drew inspiration from three sources: First, the pioneering work of Deligne on Shimura varieties, which discovered the Hodge-theoretic significance of minuscule cocharacters. Second, the work of Griffiths-Schmid -- revisited by Carayol and Green-Griffiths-Kerr -- on moduli spaces of Hodge structures, which attempted to pierce the barrier of minuscule cocharacters. Third, the work of Moonen-Wedhorn and Pink-Wedhorn-Ziegler on $G$-Zips, which introduced a new class of geometric objects associated to cocharacters of reductive groups.  We thank the authors of these works for their encouragement and for enlightening discussions related to this paper.

We are grateful to Torsten Wedhorn for helpful comments on an earlier draft of this paper.
In addition, it is a pleasure to thank Yohan Brunebarbe, Marc-Hubert Nicole, Stefan Patrikis, Colleen Robles, Sug Woo Shin, Benoit Stroh and David Vogan for helpful discussions. Finally, we thank the referee for his/her careful reading of the paper.

\section{Notation and structure theory} Let $(G,T)$ as in \S\ref{sec-intro-quasi-constant}, with root datum~\eqref{eq-root-datum-G}.
\label{sec-notation}
\subsection{Structure theory} \label{sec-set-up-classification}
\subsubsection{Simply-connected covering and adjoint projection}
\label{sec-simply-conn-adjoint}
Write $G^{\der}$ (resp. $G^{\ad}$) for the derived subgroup (resp. adjoint quotient) of $G$ and $\tilde G$ for the simply-connected cover of $G^{\der}$. Write $\pr: G \twoheadrightarrow G^{\ad}$ for the natural projection and $s:\tilde G \to G$  for the "quasi-section" of $\pr$, composition of the projection $\tilde G \twoheadrightarrow G^{\der}$ with the inclusion $G^{\der} \hookrightarrow G$.

The root datum~\eqref{eq-root-datum-G} canonically induces ones of $G^{\der}, \tilde G$ and $G^{\ad}$ as follows (see \cite[Part II, 1.18]{jantzen-representations}): Let $$X_0^*(T):=\{\chi \in X^*(T)| \langle \chi ,\alpha^{\vee} \rangle =0\mbox{ for all } \alpha \in \Phi\} \mbox{ and }
T^{\der}:=\bigcap_{\chi \in X_0^*(T)} \ker \chi.$$
Then $T^{\der}$ is a maximal torus in $G^{\der}$ and 
$(X^*(T^{\der}), \Phi; X_*(T^{\der}), \Phi^{\vee}) $ is the root datum of $(G^{\der}, T^{\der})$, where the roots $\Phi$ are restricted to $T^{\der}$.  
Let $\tilde T$ (resp. $T^{\ad}$) denote the preimage of $T^{\der}$ in $\tilde G$ (resp. the image of $T^{\der}$ in $G^{\ad}$). Then $\tilde T$ and $T^{\ad}$ are maximal tori in $\tilde G$ and $G^{\ad}$ respectively; the roots (resp. coroots) of the three pairs $(\tilde G, \tilde T), (G^{\der}, T^{\der}),(G^{\ad}, T^{\ad})$ are identified via the central isogenies $(\tilde G, \tilde T) \to (G^{\der}, T^{\der}) \to(G^{\ad}, T^{\ad})$.

\subsubsection{Decompositions over an algebraically closed field}
\label{sec-decomp-alg-closed}
Let $K$ be an algebraically closed field extension of $k$.
Over $K$, one has the decompositions \begin{equation} \tilde G_{K} \cong \prod_{i=1}^d \tilde G_i \mbox{ \hspace{.2in} and \hspace{.2in} } G^{\ad}_{K} \cong \prod_{i=1}^d G_i^{\ad} , \label{eq factors of G}\end{equation} 
where each $\tilde G_i$ is a simple, simply-connected $K$-group and $G_i^{\ad}$ is its adjoint group.  
Set $s_i:\tilde G_i \to G_{K}$ (resp. $\pr^i: G_K \to G_i^{\ad}$) for the composition of $s$ (resp. $\pr$) with the embedding along (resp. projection onto) the $i$th component by means of~\eqref{eq factors of G}.

In view of~\eqref{eq factors of G}, one has
\begin{equation} \tilde T_{K} \cong \prod_{i=1}^d \tilde T_i \mbox{ \hspace{.2in} and \hspace{.2in} } T^{\ad}_{K} \cong \prod_{i=1}^d T_i^{\ad} , \label{eq factors of T}\end{equation} where $ \tilde T_i \subset \tilde G_i$, $T_i^{\ad} \subset G_i^{\ad}$ are maximal tori and $\tilde T_i$ is the inverse image of $T_i^{\ad}$ under the projection $\tilde G_i \twoheadrightarrow G_i^{\ad}$.
\subsubsection{Dynkin diagram} 
\label{sec-dynkin}
Fix a basis of simple roots $\Delta \subset \Phi$. Write $\Delta=\bigsqcup_{i=1}^d \Delta_i$ and $\Phi=\bigsqcup_{i=1}^d \Phi_i$ for the decompositions of $\Delta$ and $\Phi$ corresponding to~\eqref{eq factors of G}. Denote by $\Dfr$ (resp. $\Dfr_i$) the Dynkin diagram of $\Delta$ (resp. $\Delta_i$). Given $\alpha \in \Delta$, write $v_{\alpha}$ for the corresponding vertex of $\Dfr$. Recall that $\Dfr$ is called \underline{simply-laced} if no two vertices of $\Dfr$ are joined by more than one edge (equivalently all roots have the same length); otherwise we say $\Dfr$ is \underline{multi-laced}.
\subsubsection{(Co)Root multiplicities}
\label{sec-root-mult}
When $G_{K}$ is simple,  write $^h \! \alpha$ (resp. $^h\!\alpha^{\vee}$) for the highest root (resp. highest coroot). Let $({}^h\!\alpha)^{\vee}$ be the coroot corresponding to the highest root under the bijection $\Phi \to \Phi^{\vee}$, $\alpha \mapsto \alpha^{\vee}$. Beware that ${}^h \! \alpha^{\vee} \neq ({}^h\!\alpha)^{\vee}$ precisely when $\Phi$ is multi-laced.   One has decompositions into simple (co)roots $$
^h\!\alpha=\sum_{\alpha \in \Delta} m(\alpha) \alpha \hspace{1cm} \mbox{ and } \hspace{1cm }^h\!\alpha^{\vee}=\sum_{\alpha \in \Delta} m^{\vee}(\alpha) \alpha^{\vee}$$ with $m(\alpha), m^{\vee}(\alpha) \in \zgeqo$ for all $\alpha \in \Delta$.  Recall that a vertex $v_{\alpha}$ of $\Dfr$ is called \uline{special} if $\alpha$ satisfies $m(\alpha)=1$. Say that $v_{\alpha}$ is \uline{co-special} if $m^{\vee}(\alpha)=1$.
\subsubsection{Fundamental (Co)weights}
\label{sec-fund-weights} Suppose $G$ is semisimple. Then the set of simple roots $\Delta$ (resp. simple coroots $\Delta^{\vee}$) is a basis of $X^*(T)_{\QQ}$ (resp. $X_*(T)_{\QQ}$). For $\alpha \in \Delta$, write $\eta(\alpha) \in X^*(T)_{\QQ}$ (resp. $\eta(\alpha^{\vee})\in X_*(T)_{\QQ}$) for the corresponding fundamental weight (resp. fundamental coweight) defined by \[\langle \eta(\alpha), \beta^{\vee} \rangle=\langle \beta, \eta(\alpha^{\vee}) \rangle =\left\{ \begin{array}{ccc} 1 & \mbox{ if } & \beta = \alpha \\
0 & \mbox{ if } & \beta \in \Delta, \beta \neq \alpha
 \end{array} \right. .\]
\subsection{Minuscule and cominuscule (co)characters}

\subsubsection{Minuscule (co)characters} 
\label{sec-minuscule}
Let $\chi \in \chargpt$ and $\mu \in \cochargpt$. Recall that $\chi$ (resp. $\mu$) is \underline{minuscule} if, for every root $\alpha$, one has $\langle \chi , \alpha^{\vee} \rangle \in \{0,1,-1\}$ (resp. $\langle \alpha, \mu \rangle \in \{0,1,-1\}$). Note the resemblance with the definition of quasi-constant (co)characters (\Def~\ref{def quasi constant}). 

\subsubsection{Cominuscule (co)characters}
\label{sec-cominuscule}

Suppose $G$ is semisimple. Then the literature also contains a far less standard (and arguably less natural, see
\S\ref{sec-motivation-unification-co-min}) notion of cominuscule (co)character.
Following \cite[\Def~9.0.14]{Billey-Lakshmibai-book},  $\chi \in X^*(T)$ is termed \underline{cominuscule} if there exists a basis $\Delta \subset \Phi$ of simple roots such that \begin{enumerate}
\item $\chi=\eta(\alpha)$ for some $\alpha \in \Delta$, and
\item the fundamental coweight $\eta(\alpha^{\vee})$ is minuscule.\footnote{Our definition is an equivalent variant of the one given in \loccitn}
\end{enumerate}
A cominuscule cocharacter is defined by replacing `roots' with `coroots' and `fundamental weights' with `fundamental coweights'.

\subsubsection{Relation to fundamental (co)weights}
\label{sec-fund-minusc-cominusc}
The notions of
\S\S\ref{sec-dynkin}--\ref{sec-fund-weights} and those just recalled in \S\S\ref{sec-minuscule}--\ref{sec-cominuscule} are linked as follows: A fundamental weight $\eta(\alpha)$ is minuscule (resp. cominuscule) if and only if the vertex $v_{\alpha}$ of $\Dfr$ is cospecial (resp. special). Dually, a fundamental coweight $\eta(\alpha^{\vee})$ is minuscule (resp. cominuscule) if and only if $v_{\alpha}$ is special (resp. cospecial).
 
\section{Classification and duality}
\label{sec-classification}
\S\ref{sec-proof-simple}~is devoted to the proof of \Th~\ref{th-quasi-constant-cominuscule}. We treat the case of characters; the case of cocharacters is completely analogous and left as an exercise. Following some preliminaries, the proof is divided into two cases, according to whether the Dynkin diagram $\Dfr$ is simply-laced or not. The simply-laced case is much simpler. In the multi-laced case, the crux of the argument is to show that a $\Delta$-dominant, quasi-constant character is a multiple of a fundamental weight, see Lemma~\ref{lem fundamental}.

\S\ref{sec-proof-generalization} deduces \Th~\ref{th-gen-quasi-constant-cominuscule} from the special case given by \Th~\ref{th-quasi-constant-cominuscule}. The duality between quasi-constant characters and cocharacters is described in \S\ref{sec-duality}.

\subsection{The absolutely simple and simply-connected case}
\label{sec-proof-simple}
Throughout \S\ref{sec-proof-simple}, suppose $k$ is algebraically closed and that the $k$-group $G$ is simple and simply-connected. 
Consequently, the fundamental weight $\eta(\alpha) \in X^*(T)$ for all $\alpha \in \Delta$.

Assume $\chi \in \chargpt$ is quasi-constant and nontrivial.  Without loss of generality, we may assume that $\chi$ is $\Delta$-dominant. Write $\chi$ as a linear combination of fundamental weights  \begin{equation} \label{eq chi sum fund weights} \chi=\sum_{\alpha \in \Delta} m_{\alpha}(\chi)\eta(\alpha) \end{equation} with $m_{\alpha}(\chi) \in \zgeqz$ for all $\alpha \in \Delta$. Using \S\ref{sec-root-mult}, put \begin{equation}
\label{eq def M of chi} M(\chi)=\sum_{\alpha \in \Delta} m^{\vee}(\alpha) m_{\alpha}(\chi). \end{equation} For all $\alpha \in \Delta$, one has $\langle \chi, \alpha^{\vee} \rangle=m_{\alpha}(\chi)$  and $\langle \chi, ^h\! \alpha^{\vee} \rangle=M(\chi)$. Since $\chi \neq 0$, there exists $\beta \in \Delta$ such that $m_{\beta}(\chi)>0$. Fix such a $\beta$ for the rest of \S\ref{sec-proof-simple}.

\begin{proof}[Proof of \textnormal{\Th~\ref{th-quasi-constant-cominuscule}}, simply-laced case] Assume the Dynkin diagram $\Dfr$ is simply-laced; equivalently $W$ acts transitively on both $\Phi$ and $\Phi^{\vee}$.
 In particular, all the simple coroots and the highest coroot are in the same W-orbit.

Since $\chi$ is quasi-constant and $m_{\beta}(\chi),M(\chi)>0$, one has $m_{\beta}(\chi)=M(\chi)$. As $m_{\alpha}(\chi) \geq 0$ and $m^{\vee}(\alpha) \geq 1$ for all $\alpha \in \Delta$, we deduce from~\eqref{eq def M of chi} that $m_{\alpha}(\chi)=0$ for all $\alpha\neq \beta$ and $m^{\vee}(\beta)=1$. Therefore $\chi=m_{\beta}(\chi) \eta(\beta)$. Finally, $m^{\vee}(\beta)=1$ means that the vertex $v_{\beta}$ of $\Dfr$ is cospecial
(\S\ref{sec-root-mult}); equivalently $\eta(\beta)$ is minuscule (\S\ref{sec-fund-minusc-cominusc}). \end{proof}
\begin{proof}[Proof of \textnormal{\Th~\ref{th-quasi-constant-cominuscule}}, multi-laced case]
Assume for the rest of the proof that $\Dfr$ is multi-laced (so $G$ is of type $B_n$, $C_n$, $G_2$ or $F_4$, $n\geq 2$). Then $\Phi$ (resp. $\Phi^{\vee}$) is the (disjoint) union of two Weyl group orbits; two roots (resp. coroots) are in the same orbit if and only if they have the same length.

\begin{lemma} Assume $\chi \in X^*(T)$ is quasi-constant and $\Delta$-dominant.  Then $\chi$ is a multiple of a fundamental weight. \label{lem fundamental}\end{lemma}
\begin{proof} Suppose the conclusion does not hold. Then, in addition to $m_{\beta}(\chi)>0$, there must exist $\gamma \in \Delta$, distinct from $\beta$, such that $m_{\gamma}(\chi)>0$. Since $\Dfr$ is multi-laced, it admits at most one minuscule fundamental weight (zero for $G_2$ and $F_4$, one for $B_n$ and $C_n$, $n\geq 2$). Therefore at least one of $\eta(\beta)$ and $\eta(\gamma)$ is not minuscule.

Without loss of generality, we may assume $\eta(\beta)$ is not minuscule. Equivalently, $v_{\beta}$ is not cospecial, or what amounts to the same, $\langle \eta(\beta), ^h\! \alpha^{\vee} \rangle >1$.

Let $M=\sum_{\alpha \in \Delta}m^{\vee}(\alpha)$. (In terms of~\eqref{eq def M of chi}, one has $M=M(\rho)$, where $\rho$ is the half-sum of the positive roots.)
One knows that the highest coroot $^h\!\alpha^{\vee}$ can be written as a sum of simple coroots $^h\!\alpha^{\vee}=\sum_{i=1}^M\alpha_i^{\vee}$ such that every partial sum $S_{M'}=\sum_{i=1}^{M'}\alpha_i^{\vee}$, ($1\leq M' \leq M$) is a coroot \cite[II.12, Problem 7]{Knapp-beyond-intro-book}.

We claim that there exists a positive coroot $\delta^{\vee}$ whose decomposition into simple coroots either (i) involves both $\beta^{\vee}$ and $\gamma^{\vee}$ with $m^{\vee}(\beta)=1$ and $m^{\vee}(\gamma)\geq 1$, or (ii) involves $\beta^{\vee}$ with multiplicity $\geq 2$ and does not involve $\gamma^{\vee}$. Indeed, the largest partial sum $S_M=  {}^h\!\alpha^{\vee}$ contains $\beta^{\vee}$ with multiplicity $m^{\vee}(\beta) \geq 2$ and $\gamma^{\vee}$ with $m^{\vee}(\gamma) \geq 1$. Let $M^*$, $2 \leq M^* \leq M$ be the smallest integer such that the partial sum $S_{M^*}$ has the same property. Then $S_{M^*-1}$ is a coroot which satisfies the claim.

Put $\delta^{\vee}=S_{M^*-1}$. By construction, one has the sequence of inequalities 
$$ 
\langle \chi, ^h\!\alpha^{\vee} \rangle > \langle \chi , \delta^{\vee} \rangle> \langle \chi, \beta^{\vee} \rangle >0. 
$$ 
Hence the pairings of $\chi$ with coroots take on at least 3 strictly positive values. Since the coroots form two Weyl group orbits, there exists a $W$-orbit whose pairing with $\chi$ takes on at least two strictly positive values. Thus $\chi$ is not quasi-constant.
\end{proof}
It is left to show that  fundamental weights which are neither minuscule, nor cominuscule are not quasi-constant, by using the tables cited in \Rmk~\ref{rmk-cominuscule}. This is done case-by-case in 
Lemmas~\ref{lem-G2},~\ref{lem-F4} and~\ref{lem-BC} below. Let $e_i$ denote the $i$th coordinate vector in $\ZZ^k$.

\begin{lemma}
\label{lem-G2}
If $G$ has type $G_2$, then $T$ admits no quasi-constant characters. 
\end{lemma}

\begin{proof}  Let $\alpha_1=e_1-e_2$ and $\alpha_2=-2e_1+e_2+e_3$.  Following \cite[\Chap~VI, planche ~IX]{bourbaki-lie-4-6}, choose an identification of the root datum of $(G, T)$ so that $X^*(T)_{\QQ}=X_*(T)_{\QQ}=\{(x_1,x_2,x_3) \in \QQ^3|x_1+x_2+x_3=0\}$, $\Delta=\{\alpha_1, \alpha_2\}$ and  $\langle, \rangle$ is the standard inner product on $\QQ^3$ restricted to $X^*(T)_{\QQ}$. Then the Weyl group orbit of long coroots is \[O_3=\{\pm(e_1-e_2), \pm(e_1-e_3), \pm (e_2-e_3)\}\] and the orbit of short coroots is \[O_1=\{\pm\frac{1}{3}(2e_1-e_2-e_3), \pm\frac{1}{3}(2e_2-e_1-e_3), \pm\frac{1}{3}(2e_3-e_1-e_2)\}.\]
Moreover, $\eta(\alpha_1)=e_3-e_2$, $\eta(\alpha_2)=2e_3-e_1-e_2$. 
 The computation 
$$
\{ \ 
| \langle \eta(\alpha_1), \gamma^{\vee} \rangle | \ | \ \gamma^{\vee} \in O_3
\}=
\{1,2\}=
\{ \ 
| \langle \eta(\alpha_2), \gamma^{\vee} \rangle | \ | \ \gamma^{\vee} \in O_1
\}
$$
shows that neither $\eta(\alpha_1)$, nor $\eta(\alpha_2)$ is quasi-constant.
 \end{proof}

\begin{lemma} If $G$ has type $F_4$, then $T$ admits no quasi-constant characters. 
\label{lem-F4} \end{lemma}
\begin{proof} As in \cite[\Chap~VI, planche~VIII]{bourbaki-lie-4-6}, set $\alpha_1=e_2-e_3$, $\alpha_2=e_3-e_4$, $\alpha_3=e_4$ and $\alpha_4=(e_1-e_2-e_3-e_4)/2$ in $\QQ^4$.  Choose an identification of the root datum of $(G, T)$ so that $X^*(T)_{\QQ}=X_*(T)_{\QQ}=\QQ^4$, $\Delta=\{\alpha_1, \alpha_2, \alpha_3, \alpha_4\}$ and $\langle, \rangle$ is the standard inner product on $\QQ^4$. The two Weyl group orbits of short and long coroots are respectively   \[O_1=\{\pm e_i \pm e_j | 1 \leq i \neq j \leq 4\} \mbox{ and  } O_2=\{\pm 2e_i | 1 \leq i \leq 4\} \cup \{ \pm e_1 \pm e_2 \pm e_3 \pm e_4\} .\] The fundamental weights are:  $\eta(\alpha_1)=e_1+e_2$, $\eta(\alpha_2)=2e_1+e_2+e_3$, $\eta(\alpha_3)=(3e_1+e_2+e_3+e_4)/2$, $\eta(\alpha_4)=e_1$. The computations 
$$
\{
\ | \langle \eta(\alpha_1), \alpha^{\vee} \rangle | \ | \ \alpha^{\vee} \in O_1
\}
=
\{ \ 
| \langle \eta(\alpha_3), \alpha^{\vee} \rangle | \ | \ \alpha^{\vee} \in O_1
\}
=
\{ \
| \langle \eta(\alpha_4), \alpha^{\vee} \rangle |  \ | \ \alpha^{\vee} \in O_2
\}=
\{0,1,2\}
$$ 
and 
$$
\{ \ 
| \langle \eta(\alpha_2), \alpha^{\vee} \rangle | \ | \ \alpha^{\vee} \in O_1
\}=
\{0,1,2,3\}.
$$
show that none of the fundamental weights $\eta(\alpha_i)$ ($1 \leq i \leq 4$) are quasi-constant.
 \end{proof}
\begin{lemma}
\label{lem-BC}
 Suppose $G$ is of type $B_n$ or $C_n$ ($n \geq 2$). Then the quasi-constant characters of $T$  are precisely the multiples of the two fundamental weights corresponding to the extremities of the Dynkin diagram $\Dfr$. \end{lemma}
 \begin{proof} Let \begin{equation}
\begin{array}{l} O_1=\{\pm e_i \pm e_j | 1 \leq i < j \leq n\}, \\ O_2=\{\pm 2e_i| 1\leq i \leq n\}, \\
 O_{1/2}=\{\pm e_i|1 \leq i \leq n\}. \end{array} \label{eq_roots_type_B-C}\end{equation} Identify $X^*(T)_{\QQ}$ and $X_*(T)_{\QQ}$ with $\QQ^n$ in such a way that $\Phi=O_1 \cup O_2$, $\Phi^{\vee}=O_1 \cup O_{1/2}$ in type $C_n$ and $\Phi=O_1 \cup O_{1/2}$, $\Phi^{\vee}=O_1 \cup O_2$ in type $B_n$. In each of the above four cases, the two $W$-orbits are $O_1$ and $O_j$, with $j \in \{2, 1/2\}$. In case $C_n$ (resp. $B_n$), choose $\Delta=\{e_i-e_{i+1}\}_{i=1}^{n-1} \cup\{2e_n\}$ (resp. $\Delta=\{e_i-e_{i+1}\}_{i=1}^{n-1} \cup
 \{e_n\}$). Then the fundamental weights are given by \[\begin{array}{cll} \eta(e_j-e_{j+1})& = & \sum_{i=1}^j e_i  \mbox{ for } 1\leq j \leq n-1 \\
\eta(e_n) & = & (\sum_{i=1}^ne_i)/2 \\
\eta(2e_n) & = & \sum_{i=1}^ne_i.
\end{array}\]   In both cases  $B_n$ and $C_n$, when $n \geq 3$ and $1 <j < n$, one has \[\{ \ |\langle \eta(e_j-e_{j+1}), \alpha^{\vee} \rangle | \ | \ \alpha^{\vee} \in O_1\}=\{0,1,2\}.\] Hence $\eta(e_j-e_{j+1})$ is not quasi-constant for all $j$, $1<j <n$.
 \end{proof}
 Since all multi-laced cases $B_n$, $C_n$, $G_2$ and $F_4$ have been treated, the proof of \Th~\ref{th-quasi-constant-cominuscule} is complete.
\end{proof}

\subsection{Classification II: The general case } 
\label{sec-proof-generalization}

\begin{lemma} \label{lem_quasi-constant_weyl-galois_closed} Each of the three properties `minuscule', `cominuscule' and `quasi-constant' is closed under the action of $W \rtimes \galk$. In other words, suppose $\chi \in X^*(T)$ and $\sigma \in W \rtimes \galk$. Then $\chi$ is minuscule (resp. cominuscule, quasi-constant)  if and only if $\sigma \chi$ is minuscule (resp. cominuscule, quasi-constant).
\begin{proof}
The action of $W \rtimes \galk$ is orthogonal relative to the perfect pairing $\langle, \rangle$. Hence \[\langle \sigma \chi , \alpha^{\vee} \rangle=\langle \chi , \sigma^{-1}\alpha^{\vee} \rangle\]
for all $\alpha \in \Phi$. The result follows.
\end{proof}

\end{lemma}
\begin{proof}[Proof of \textnormal{\Th~\ref{th-gen-quasi-constant-cominuscule}\ref{item-red-k-simple}}] It is clear that $\chi \in X^*(T)$ is quasi-constant for $(G,T)$ if and only if $s^*(\chi) \in X^*(\tilde{T})$ is quasi-constant for $(\tilde{G}, \tilde{T})$. The pair $(\tilde{G}, \tilde{T})$ decomposes as a product of pairs $(H_j, S_j)$, where each $H_j$ is $k$-simple, and $S_j$ is a maximal torus of $H_j$ defined over $k$. The $X^*(S_j)$ are stable under the action of $W \rtimes \galk$. Consequently, a character of $\tilde T$ is quasi-constant if and only if its pullback to $S_j$ is so for every $j$. 
\end{proof}

Given that \Th~\ref{th-gen-quasi-constant-cominuscule}\ref{item-red-k-simple} has been proved, it will be assumed for the rest of \S\ref{sec-proof-generalization} that $G$ is $k$-simple.

\begin{proof}[Proof of \textnormal{\Th~\ref{th-gen-quasi-constant-cominuscule}\ref{item-k-simple}} ,``$\Longrightarrow$'': ]
Suppose $\chi \in X^*(T)$ is quasi-constant.  We show that $\chi$ satisfies (i)-(iii) of
~\ref{item-k-simple}. Without loss of generality, we may assume $\chi$ is $\Delta$-dominant.

For all $i$ ($1\leq i \leq d$), the pullback $s_i^*(\chi) \in X^*(\tilde T_i)$ is quasi-constant for $\tilde G_i$. Define $\xi_i \in X^*(\tilde T_i)$ as follows: If $s_i^*(\chi)=0$,  set $\xi_i=0$. Otherwise, \Th~\ref{th-quasi-constant-cominuscule} yields $c_i \in \ZZ_{\geq 1}$ such that $s_i^*(\chi)/c_i$ is either minuscule or cominuscule; set $\xi_i=s_i^*(\chi)/c_i$. In this case, there exists $\alpha_i \in \Delta_i$ such that $\xi_i=\eta(\alpha_i)$, see \S\ref{sec-fund-minusc-cominusc}.

For every pair $(i,j)$ with $\xi_i\neq 0$ and $\xi_j \neq 0$, it remains to show that $c_i=c_j$ and that $\xi_i,\xi_j$ are either both minuscule or both cominuscule. Since $G$ is $k$-simple, $\galk$ acts transitively on $\{\Dfr_s\}_{s=1}^d$ (see \S\ref{sec-dynkin}). In particular, the Dynkin diagrams $\Dfr_1, \ldots, \Dfr_d$ are pairwise isomorphic (and so too are the groups $\tilde G_1, \ldots \tilde G_d$, as they are simply-connected). 
Fix a pair $(i,j)$ with $\xi_i, \xi_j \neq 0$ and $\sigma \in \galk$ mapping $\Dfr_i$ to $\Dfr_j$.

Assume first that all the $\Dfr_s$ are simply-laced. Then $\Phi_j$ forms a single Weyl group orbit. Thus $\sigma \alpha_i \in W\alpha_j$, so $\alpha_i$ and $\alpha_j$ are conjugate under $W \rtimes \galk$. Moreover, since $\Dfr_i$ and $\Dfr_j$ are simply-laced, both $\xi_i$ and $\xi_j$ are minuscule (\Rmk~\ref{rmk-cominuscule}). Finally, $c_i=c_j$, for otherwise the set 
$\{ \ 
|\langle \chi, \tau \alpha_i^{\vee}\rangle | \ | \ \tau \in W \rtimes \galk
\}
$
would contain the two nonzero distinct values 
$c_i, c_j$ 
(and $0$ when 
$|\Dfr_i|>1$
).

We are left with the case that neither $\Dfr_i$ nor $\Dfr_j$ is simply-laced. So each of $\Dfr_i$ and $\Dfr_j$ admits no non-trivial automorphisms. Hence  an element of $\galk$ which maps $\Dfr_i$ to itself (as a set) must in fact fix it pointwise.

Moreover, by \Rmk~\ref{rmk-cominuscule}, either both $\Dfr_i$ and $\Dfr_j$ are of type $B_n$, or both are of type $C_n$. In each of the cases $B_n$ and $C_n$, one extremity of the Dynkin diagram is special but not cospecial, while the other is cospecial but not special. All of the other vertices in types $B_n$ and $C_n$ are neither special nor cospecial.

We claim that either $\alpha_i$ and $\alpha_j$ are both special, or both cospecial. Assume for a contradication that this is not the case. By symmetry we may assume that $\alpha_i$ is special and $\alpha_j$ is cospecial.  Using the notation~\eqref{eq_roots_type_B-C}, for $j \in \{1/2, 1, 2\}$, let $\tilde O_j$ be a $W \rtimes \galk$-orbit of coroots which identifies with $O_j$ on both the $i$th and $j$th factors.

In case $C_n$, one has \begin{subequations} \begin{equation}
\{ \
|\langle \chi, \alpha^{\vee} \rangle | \ | \ \alpha^{\vee} \in \tilde O_1
\}=
\{0,2c_i, c_j\}, \label{eq_a_Cn_not_quasi_constant}\end{equation}
\begin{equation}
\{ \ 
|\langle \chi, \alpha^{\vee} \rangle | \ |  \ \alpha^{\vee} \in \tilde O_{1/2}
\}=
\{0,c_i, c_j\}. \label{eq_b_Cn_not_quasi_constant}\end{equation}  \end{subequations} Since $\chi$ is quasi-constant,~\eqref{eq_a_Cn_not_quasi_constant} implies $2c_i=c_j$, while~\eqref{eq_b_Cn_not_quasi_constant} implies $c_i=c_j$. This is a contradiction since $c_i \neq 0$ and $c_j \neq 0$ by assumption. The same contradiction is reached in case $B_n$, where $(O_1, O_{1/2})$ is replaced by $(O_2, O_{1})$. 
This contradiction proves the claim.

By Lemma~\ref{lem_quasi-constant_weyl-galois_closed}, $\sigma$ maps the unique special (resp. cospecial) vertex of $\Dfr_i$ to the unique special (resp. cospecial) vertex of $\Dfr_j$. Together with claim that was just established, this shows that 
$\sigma \alpha_i=\alpha_j$.

Finally, $\sigma \alpha_i=\alpha_j$
implies that 
$
\{ \ 
|\langle \chi, \alpha^{\vee} \rangle| \ | \ \alpha^{\vee} \in \tilde O_1
\}
$
is equal to either 
$\{0, c_i, c_j\}$, 
$\{0, 2c_i, 2c_j\}$ 
or $\{c_i,c_j\}$. (\footnote{Although it is not used in the proof, we note that the third possibility $\{c_i,c_j\}$ can only occur in type $B_2=C_2$.}). Since $\chi$ is quasi-constant, we conclude either way that $c_i=c_j$. This completes the proof that $\chi$ satisfies conditions (i)-(iii) of \Th~\ref{th-gen-quasi-constant-cominuscule}\ref{item-k-simple}
\end{proof}

\begin{proof}[Proof of \textnormal{\Th~\ref{th-gen-quasi-constant-cominuscule}\ref{item-k-simple}} ,``$\Longleftarrow$'': ]
Conversely, suppose that $\chi \in X^*(T)$ and that $s^*\chi=m(\xi_1, \ldots,x_d)$, where $m, \xi_1, \ldots, \xi_d$ satisfy (i)-(iii) of \ref{item-k-simple}. We need to check that $\chi$ is quasi-constant. 

Assume $\sigma \in \galk$,  $\alpha \in \Phi$ and  $\langle \chi , \alpha^{\vee} \rangle, \langle \chi , \sigma \alpha^{\vee} \rangle \neq 0$.
We have to show that $|\langle \chi , \alpha^{\vee} \rangle | = | \langle \chi , \sigma \alpha^{\vee} \rangle |$.
 Let $i, j \in \{1,2, \ldots, d\}$ such that $\alpha \in \Dfr_i$ and $\sigma \alpha \in \Dfr_j$ (the possibility $i=j$ is not excluded).

Since $\langle \chi , \alpha^{\vee} \rangle, \langle \chi , \sigma \alpha^{\vee} \rangle \neq 0$, $\xi_i$ and $\xi_j$ are both nontrivial. By condition (iii) of~\ref{item-k-simple}, $\xi_i$ and $\xi_j$ are either both minuscule or both cominuscule. By \Rmk~\ref{rmk-cominuscule}, $\xi_i=\eta(\alpha_i)$ and $\xi_j=\eta(\alpha_j)$ for some $\alpha_i \in \Delta_i$ and $\alpha_j \in \Delta_j)$. 

Suppose first that $\xi_i$ and $\xi_j$ are both minuscule. Then $$|\langle \chi , \alpha^{\vee} \rangle | =m |\langle \xi_i, \alpha^{\vee} \rangle |=m  =m |\langle \xi_j, \sigma \alpha^{\vee} \rangle|= | \langle \chi , \sigma \alpha^{\vee} \rangle |.$$

Now assume $\xi_i$ and $\xi_j$ are both cominuscule. Since $G$ is $k$-simple,  $\Dfr_i$ and $\Dfr_j$ are either both of type $B_n$ or both of type $C_n$. One checks directly using~\eqref{eq_roots_type_B-C} that in type $C_n$ both
$|\langle \eta(\alpha_i), \alpha^{\vee} \rangle|$ and $|\langle \eta(\alpha_j), \sigma\alpha^{\vee} \rangle|$ are equal to $1$ (resp. $2$) if $\alpha^{\vee}$ belongs to the $W \rtimes \galk$ orbit $\tilde O_{1/2}$ (resp. $\tilde O_{1}$) and that in type $B_n$  both
$|\langle \eta(\alpha_i), \alpha^{\vee}\rangle|$ and $|\langle \eta(\alpha_j), \alpha^{\vee} \rangle|$ are equal to $1$ (resp. $2$) if $\alpha^{\vee}$ belongs to $\tilde O_{1}$ (resp. $\tilde O_{2}$).
\end{proof}

\subsection{Duality}
\label{sec-duality} Here we explain the duality  between quasi-constant characters and cocharacters of semisimple $G$, see Construction~\ref{const-duality} . The key properties of the construction follow directly from the classification and are provided in \Prop~\ref{prop-duality}. 

If $\rfr \subset X^*(T)_{\QQ}$ (resp. $\rfr \subset X_*(T)_{\QQ}$ whose image is not contained in the center of $G$)
is a quasi-constant ray (\Def~\ref{def-quasi-constant lines}), then $s_i^*(\rfr)$ (resp. $\pr^i_*(\rfr)$) is a quasi-constant ray in $X^*(\tilde T_i)$ (resp. $X_*(T_i^{\ad})$). 
\begin{construction} \label{const-duality} Let $\rfr$ be a quasi-constant ray in $X_*(T)_{\QQ}$.  We construct a ``dual quasi-constant ray'' $\rfr^{\vee} \subset X^*(T)_{\QQ}$.
 By \Th~\ref{th-gen-quasi-constant-cominuscule}\ref{item-k-simple} and \S\ref{sec-fund-minusc-cominusc}, there exists a basis of simple roots $\Delta \subset \Phi$ such that, for every $i$ ($1 \leq i \leq d$), $\pr^i_*(\rfr)$ is either trivial or contains a fundamental coweight $\eta(\alpha_i^{\vee})$ for some $\alpha_i \in \Delta_i$.  Let $\rfr_{\ad}^{\vee}$ be the ray in $X^*( T^{\ad})_{\QQ}=\prod_{i=1}^nX^*( T^{\ad}_i)_{\QQ}$ spanned by the vector $\chi=(\chi_i)_{i=1}^n$ whose $i$th coordinate is defined by
\begin{equation}\chi_i=\left\{ \begin{array}{clc} \eta(\alpha_i) & {\rm if } & \pr^i_*(\rfr) \neq 0 \\ 0 & {\rm if } & \pr^i_*(\rfr)=0 \end{array} \right.  \label{eq-dual-quasi-constant-line}\end{equation}
Set $\rfr^{\vee}:=\pr^*(\rfr_{\ad}^{\vee})$. 
\end{construction}
\begin{rmk} \label{rmk dual line const}
It is clear that there is a construction dual to~\ref{const-duality} which starts with a quasi-constant ray in $X^*(T)_{\QQ}$ and produces a quasi-constant ray in $X_*(T)_{\QQ}$.
\end{rmk}

\begin{proposition} 
\label{prop-duality} Construction~\textnormal{\ref{const-duality}} satisfies the following properties:
\begin{enumerate}
\item \label{item-ladvee-quasi-constant} The ray $\rfr^{\vee}_{\ad}$ in $X^*(T^{\ad})_{\QQ}$ is quasi-constant.

\item \label{item-lvee-quasi-constant} The ray $\rfr^{\vee} \subset X^*(T)_{\QQ}$ is quasi-constant.

\item \label{item-lvee-restriction-levi} The quasi-constant ray $\rfr^{\vee} \subset X^*(T)_{\QQ}$ is the restriction of a ray in $X^*(\cent(\rfr))$ (see \textnormal{\Rmk~\ref{rmk-cent-line}}).

\item \label{item-dual-line-maximal-levi} The Levi $\cent(\rfr)$ of $G$ is the maximal Levi satisfying property~\ref{item-lvee-restriction-levi}.

\item 
\label{item-duality-bijection}
If $G$ is semisimple, then $\rfr \to \rfr^{\vee}$ is a bijection between quasi-constant rays in $X_*(T)$ and those in $X^*(T)$.
\end{enumerate}
 \end{proposition}

\begin{rmk} \label{rmk-cent-line}
If $\nu \in X_*(T)$ and $m \in \ZZ \setminus \{0\}$, then $\cent(\nu)=\cent(m\nu)$. Indeed, as centralizers of subtori of $T$,  both $\cent(\nu)$ and $\cent(m\nu)$ are Levi subgroups of $G$ containing the maximal torus $T$. Thus each is determined by the subset of $\Delta$ orthogonal to the cocharacter. But for all $\alpha \in \Delta$, one has $\langle \alpha, \nu \rangle=0$ if and only if $\langle \alpha, m\nu \rangle=0$. Therefore the centralizer of a ray (or line) in $X_*(T)_{\QQ}$ is well-defined.
\end{rmk}

\begin{proof}[Proof of \textnormal{\Prop~\ref{prop-duality}}:] 
Part~\ref{item-ladvee-quasi-constant} is a direct consequence of \Th~\ref{th-gen-quasi-constant-cominuscule}\ref{item-k-simple}. A ray $\mfr \subset X^*(T^{\ad})_{\QQ}$ is quasi-constant if and only if every element of $\pr^*\mfr \cap X^*(T)$ is. This gives~\ref{item-lvee-quasi-constant}.

The combination of Parts~\ref{item-lvee-restriction-levi} and~\ref{item-dual-line-maximal-levi} is equivalent to $\langle \rfr^{\vee}, \alpha^{\vee} \rangle=0$ for $\alpha \in \Delta$ if and only if $\alpha \in \Phi(\cent(\rfr), T) \cap \Delta$ (the simple roots pertaining to the Levi $\cent(\rfr)$). By~\eqref{eq-dual-quasi-constant-line}, $\langle \rfr^{\vee}, \beta^{\vee} \rangle\neq 0$ for $\beta \in \Delta$ if and only if $\beta=\alpha_i$ for some $i$ satisfying $\pr^i_*(\rfr) \neq 0$. The latter holds if and only if $\pr^i_*(\rfr)$ contains the fundamental coweight $\eta(\alpha_i^{\vee})$. Since $\langle \alpha_i, \eta(\alpha_i^{\vee}) \rangle=1$, we deduce that $\langle \alpha_i,  \pr_*^i(\rfr) \rangle \neq 0$ (and so also $\langle \alpha_i,  \rfr \rangle \neq 0$) if and only if $\beta=\alpha_i \not \in \Phi(\cent(\rfr), T)$.

Finally, if $G$ is semisimple, then its fundamental weights (resp. fundamental coweights) furnish a basis of $X^*(T)_{\QQ}$ (resp. $X_*(T)_{\QQ}$). Thus~\ref{item-duality-bijection} follows from \Th~\ref{th-gen-quasi-constant-cominuscule}. 
\end{proof}

\section{The Hodge line bundle is quasi-constant}
\label{sec-hodge-q-constant}
This \S~proves \Th~\ref{th-intro-omega-quasi-constant}, that the Hodge line bundle is quasi-constant.
The proof relies heavily on Deligne's analysis of symplectic embeddings of Shimura data \cite[\S1.3]{Deligne-Shimura-varieties}.

As in \loccitn, throughout \S\ref{sec-hodge-q-constant} fix $\qbar$ to be the algebraic closure of $\QQ$ in $\CC$. This choice is justified by the fact that the reflex field $E$ of the Shimura datum $\gx$ is defined as a subfield of $\CC$. We use the notation of \S\ref{sec-intro-shimura} and \S\ref{sec-notation}. 
 In particular, $\Delta$ denotes the set of simple roots of $\TT_{\qbar}$ in $\GG_{\qbar}$ and $h \in \XX$ with associated cocharacter $\mu$. Let $\Phi^+$ be the system of positive roots corresponding to $\Delta$. We normalize $\mu$ so that 
$\langle \alpha, \mu \rangle=1$ 
for 
$\alpha \in \Phi$  
if and only if 
$\alpha \in \Phi^+ \setminus \Phi(\LL,\TT)$.

Let $V$ be a $2g$-dimensional $\QQ$-vector space and $Q$ a non-degenerate, $\QQ$-valued alternating form on $V$. Let $GSp(V, Q)$ be the group of symplectic similitudes of $(V, Q)$.  Let $\std:GSp(V, Q) \to GL(V)$ be the tautological representation. Set $\rho:=\std \circ \psi$, where $\psi$ is the symplectic embedding~\eqref{eq-symplectic-embedding}. 

The crux of the proof is to reduce to a question about fundamental weights by a careful analysis of the restriction of $\rho$ to the Levi $\LL$ (\S\ref{sec-res-levi}). The latter can be solved by a simple case-by-case computation (\S\ref{sec-fund-weight-lemma}). 
\subsection{Set-up of the proof}
\label{sec-set-up-hodge-line}

As usual, put $\SS=\res_{\CC/\RR}\gm$. Throughout the proof, we abbreviate $\eta_{\omega}:=\eta_{\omega}(\psi)$ for the Hodge character of $\psi$.

 The representation $\rho:=\std \circ \psi$ of $\mathbf G$ is defined over $\mathbf Q$, since it is the composition of two morphisms which are both defined over $\mathbf Q$. The composition of $\rho_{\RR}:=\rho \otimes \RR$ with  $h:\SS \rightarrow \GG_{\RR}$ yields a polarized $\mathbf R$-Hodge structure of type $\{(-1,0),(0,-1)\}$; denote it $h_{\psi}:\mathbf S \rightarrow GL(V \otimes \mathbf R)$. The pair $(V, h_{\psi})$ is a polarized $\mathbf Q$-Hodge structure.

\subsection{Restriction to the Levi}
\label{sec-res-levi}
The restriction of $\rho_{\qbar}$ to $\LL_{\qbar}$ is equal to $V^{-1,0} \oplus V^{0,-1}$, the sum of the graded pieces\footnote{Note that, in general, the two graded pieces $V^{-1,0}, V^{0,-1}$ are not irreducible as $\LL_{\qbar}$-representations. However, they are irreducible in the special case $\gx=\shdagsp$.} of the $\RR$-Hodge structure $(V_{\RR}, h_{\psi})$. The character $-\eta_{\omega}$ is the determinant of the $\LL_{\qbar}$-representation $V^{-1,0}$. Hence $-\eta_{\omega}$ is the sum of the $\TT_{\qbar}$-weights of $V^{-1,0}$ (counted with multiplicity).

 Let $\tilde \rho$ (resp. $\tilde V^{-1,0}$, $\tilde V^{0,-1}$) be the pullback of $\rho$ (resp. $V^{-1,0}$, $V^{0,-1}$) to $\tilde \GG$ (resp. $\tilde \LL_{\qbar}$). One has \[\res^{\tilde \GG_{\qbar}}_{\tilde \LL_{\qbar}} \tilde \rho_{\qbar}=\tilde V^{-1,0} \oplus \tilde V^{0,-1}. \]
The $\tilde \LL_{\qbar}$-representations $\tilde V^{-1,0}$ and $\tilde V^{0,-1}$ are dual to one another via the alternating form $Q$.

Since $\mu$ is minuscule, so is its projection $\pr_*\mu$ to $\GG^{\ad}$. Let $\tilde \mu$ be the fractional lifting ("rel\`evement fractionaire", \cite[1.3.4]{Deligne-Shimura-varieties}) of $\pr_*\mu$ to $\tilde \GG_{\qbar}$. Let $\rho'$ be an irreducible factor of $\tilde \rho_{\qbar}$. By Lemma 1.3.5 of \loccitn, $\rho'$ has two $\tilde \mu$-weights given by $a$ and $a+1$ for some $a \in \QQ$. In other words, as $\xi$ runs through the $\tilde \TT_{\qbar}$-weights of $\tilde \rho_{\qbar}$, the pairing $\langle \xi, \tilde \mu \rangle$ takes the two values $a$ and $a+1$.
\begin{lemma} 
Let $\xi$ be a weight of $\rho'$. 
Then $\xi$ is a weight of $\tilde V^{-1,0}$
(resp. $\tilde V^{0,-1}$) 
if and only if $\langle \xi, \tilde \mu \rangle =a+1$ 
(resp. $\langle \xi, \tilde \mu \rangle=a$). \label{lem a a+1}\end{lemma}
\begin{proof} This follows easily from the proof of the aforementioned Lemma 1.3.5 of \loccitn \end{proof}

Let $\tilde L_i$ be the intersection of $\tilde G_i$ with the centralizer, in $\tilde \GG_{\qbar}$, of the fractional lifting $\tilde \mu$. Then for every $i$, either $\tilde L_i$ is the Levi of a maximal parabolic of $\tilde G_i$, or $\tilde L_i=\tilde G_i$. For every $i$ with $\tilde L_i \neq \tilde G_i$, let $\alpha_i$ be the unique simple root of $\tilde G_i$ which is not a root of $\tilde L_i$.

\begin{lemma} 
\label{lem-highest-weight}
Let $\rho'$ be an irreducible factor of $\tilde \rho_{\qbar}$ with highest weight $\xi$. Then $\xi$ is a weight of $\tilde V^{-1,0}$.
\end{lemma}
\begin{proof} Let $a$ and $a+1$ be the two $\tilde \mu$-weights of $\rho'$. Since $\rho'$ admits two distinct $\tilde \mu$-weights, it admits a $\tilde \TT_{\qbar}$-weight $\xi'$ whose pairing with $\tilde \mu$ is different from that of $\xi$ with $\tilde \mu$. By the property characterizing the highest weight, $\xi-\xi'$ is a non-negative, $\ZZ$-linear combination of simple roots. Since $\tilde \mu$ is $\Delta$-dominant, $\langle \xi-\xi', \tilde \mu \rangle \geq 0$. But by our choice of $\xi'$, one has $\langle \xi-\xi', \tilde \mu \rangle \neq  0$. Hence $\langle \xi-\xi', \tilde \mu \rangle =1$ and $\langle \xi, \tilde \mu \rangle =a+1$. So $\xi$ is a weight of $\tilde V^{-1,0}$ by Lemma~\ref{lem a a+1}.
\end{proof}

We use Lemma~\ref{lem-highest-weight} to deduce a positivity statement characterizing those weights of $\tilde \rho_{\qbar}$ which are weights of $\tilde V^{-1,0}$.

\begin{lemma}
\label{lem-hodge-filt}
Let $\xi$ be a $\tilde T_{\qbar}$-weight of $\tilde \rho_{\qbar}$.
\begin{enumerate}
\item 
\label{item-lem-hodge-filt-1}
If $\xi$ is a weight of $\tilde V^{-1,0}$, then $\langle \xi, \alpha_i^{\vee} \rangle \geq 0$ for all $i$.
\item 
\label{item-lem-hodge-filt-2}
As a partial converse, if $\langle \xi, \alpha_i^{\vee} \rangle > 0$ for some $i$, then $\xi$ is a weight of $\tilde V^{-1,0}$.
 \end{enumerate}  
 \end{lemma}

\begin{proof}
Since $\tilde \rho_{\qbar}$ is self-dual, its $\tilde \TT_{\qbar}$-weights are closed under $x \mapsto -x$. Since $\tilde V^{-1,0}$ is dual to $\tilde V^{0,-1}$, the weights of $\tilde V^{-1,0}$ are mapped bijectively onto those of $\tilde V^{0,-1}$ via $x \mapsto -x$. It follows that parts~\ref{item-lem-hodge-filt-1} and~\ref{item-lem-hodge-filt-2} of the lemma are equivalent. So assume $\xi$ is a weight of $\tilde V^{-1,0}$ and consider~\ref{item-lem-hodge-filt-1}.

Let $\rho'$ be an irreducible factor of $\tilde \rho_{\qbar}$, which admits $\xi$ as a $\tilde \TT_{\qbar}$-weight. Let $\xi_h$ be the highest weight of $\rho'$.
Since the highest weight is $\Delta$-dominant, one has $\langle \xi_h, \alpha_i^{\vee} \rangle \geq 0$. We need to use the hypothesis that $\xi$ is a weight of $\tilde V^{-1,0}$ to conclude that also $\langle \xi, \alpha_i^{\vee} \rangle \geq 0$.
Write \begin{equation} \xi_h-\xi=\sum_{\alpha \in \Delta} n(\alpha) \alpha \label{eq sum simple roots}, \end{equation} with $n(\alpha) \geq 0$ for all $\alpha \in \Delta$.

Since $\mu$ is minuscule and $\alpha_i \in \Phi^+ \setminus \Phi(\LL,\TT)$, one has $\langle \alpha_i, \mu \rangle=1$. Since $\mu=\tilde \mu \nu$ with $\nu: \gm \rightarrow \GG_{\qbar}$ fractional and central, the adjoint actions of $\mu(z)$ and $\tilde \mu (z)$ coincide. Hence also $\langle  \alpha_i, \tilde \mu \rangle=1$ and $\tilde \mu$ is $\Delta$-dominant.

Combining our assumption that $\xi$ is a weight of $\tilde V^{-1,0}$ with Lemmas~\ref{lem a a+1} and~\ref{lem-highest-weight}, we have $\langle \xi_h-\xi, \tilde \mu \rangle=0$. Therefore the multiplicity $n(\alpha_i)=0$ in~\eqref{eq sum simple roots}. A simple property of root data states that if $\langle \alpha, \beta^{\vee} \rangle >0$ for some $\alpha, \beta \in \Delta$, then $\alpha=\beta$ \cite[Lemma 2.51]{Knapp-beyond-intro-book}. Hence $\langle \xi_h-\xi, \alpha_i^{\vee} \rangle \leq 0$. But $\langle \xi_h , \alpha_i^{\vee} \rangle \geq 0$ because $\xi_h$ is $\Delta$-dominant. So  $\langle \xi, \alpha_i^{\vee} \rangle \geq 0$, as was to be shown.
\end{proof}

\subsection{Equality of fundamental weight multiplicities}
\label{sec-fund-weight-lemma}

Since $\eta_{\omega} \in X^*(\LL)$, one has $\langle \eta_{\omega}, \alpha^{\vee} \rangle =0$ for all $\alpha \in \Phi(\LL, \TT)$.  Set $s_i^*(\eta_{\omega}):=\eta_{\omega,i}$, the pullback of $\eta_{\omega}$ to $\tilde L_i$ (\S\ref{sec-decomp-alg-closed}). Suppose $\tilde L_i \neq \tilde G_i$. Then $\eta_{\omega, i}$ is a multiple, say $m_i$, of the fundamental weight $\eta(\alpha_i)$. Since $\mu$ is minuscule, $\alpha_i$ is special (\S\ref{sec-fund-minusc-cominusc}). By definition, $m_i=\langle \eta_{\omega, i}, \alpha_i^{\vee} \rangle=\langle \eta_{\omega}, \alpha_i^{\vee} \rangle$.

The next lemma shows that the multiplicities $m_i$ are constant on $\galq$-orbits.

\begin{lemma} Suppose $\tilde G_i$ is $\galq$-conjugate to $\tilde G_j$. Then $m_i=m_j$. \label{lem multiplicities indep} \end{lemma}
\begin{proof}
Let $\sigma \in \galq$ conjugate $\tilde G_i$ to $\tilde G_j$. Observe that the coroots $\sigma\alpha_i^{\vee}$  and  $\alpha_j^{\vee}$ are in the same Weyl group orbit. Indeed, $\alpha_i$ and $\alpha_j$ are both special, hence have the same length (\S\ref{sec-dynkin}). Finally, two roots are in the same Weyl group orbit if and only if the same is true of the corresponding coroots.  
Write $w\sigma\alpha_i^{\vee}=\alpha_j^{\vee}$ with $w \in W$.

Denote the set of $\tilde \TT_{\qbar}$-weights of $\tilde V^{-1,0}$ by $\Sscr$. Given $\xi \in \Sscr$, let $m(\xi)$ denote the multiplicity of $\xi$ as a weight of $\tilde V^{-1,0}$. For $\epsilon \in \{i,j\}$, let $\Sscr_{\epsilon}=\{\xi \in \Sscr | \langle \xi , \alpha_{\epsilon} \rangle >0\}$.   
Then \begin{equation}
m_{\epsilon}=\sum_{\xi \in \Sscr} m(\xi) \langle \xi, \alpha_{\epsilon}^{\vee} \rangle. 
\label{eq-m-epsilon}
\end{equation}  

By Lemma~\ref{lem-hodge-filt}\ref{item-lem-hodge-filt-1}, every summand in~\eqref{eq-m-epsilon} is nonnegative.  
Since $\Sscr \cup -\Sscr$ is the set of $\tilde \TT_{\qbar}$-weights of $\tilde \rho_{\qbar}$, 
it is closed under $x \mapsto \tau x$ for all $\tau \in W \rtimes \galq$. 
By Lemma~\ref{lem-hodge-filt}\ref{item-lem-hodge-filt-2}, the map $\Sscr \cup -\Sscr \to \Sscr \cup -\Sscr$, $x \mapsto w\sigma x$, restricts to a bijection of $\Sscr_i$ onto $\Sscr_j$. 
Another application of Lemma~\ref{lem-hodge-filt}\ref{item-lem-hodge-filt-1} gives $\Sscr_{\epsilon} \cap -\Sscr=\emptyset$. 
Thus $m(\xi)$ equals the multiplicity of $\xi$ as $\tilde {\TT}_{\qbar}$-weight of $\tilde{\rho}_{\qbar}$ for all $\xi \in \Sscr_i \cup \Sscr_j$. Hence $\xi \in \Sscr_i$ implies  
 $m(\xi)=m(w\sigma \xi)$.  Thus $m_i=m_j$.

\end{proof}

\begin{proof}[Proof of \textnormal{\Th~\ref{th-intro-omega-quasi-constant}}:] Since $\GG$ and $\tilde \GG$ have the same adjoint group, one has $\langle \eta_{\omega}, \alpha^{\vee} \rangle =\langle \tilde \eta_{\omega}, \alpha^{\vee} \rangle$ for all roots $\alpha$. It is therefore equivalent to show that $\tilde \eta_{\omega}$ is quasi-constant. Suppose a root $\alpha$ and $\sigma \in W \rtimes \galq$ satisfy $\langle \tilde \eta_{\omega}, \alpha^{\vee} \rangle \neq 0$ and $\langle \tilde \eta_{\omega}, \sigma \alpha^{\vee} \rangle \neq 0$. Let $\tilde G_i$ (resp. $\tilde G_j$) be the unique factor of $\tilde \GG_{\qbar}$ of which $\alpha$ (resp. $\sigma \alpha$) is a root. Then $\langle \tilde \eta_{\omega}, \alpha^{\vee} \rangle=\langle \tilde \eta_{\omega, i}, \alpha^{\vee} \rangle$ and $\langle \tilde \eta_{\omega}, \sigma \alpha^{\vee} \rangle =\langle \tilde \eta_{\omega, j}, \sigma \alpha^{\vee} \rangle$. By Lemma~\ref{lem multiplicities indep}, one has $\tilde \eta_{\omega, i}=m \eta(\alpha_i)$ and $\tilde \eta_{\omega, j}=m\eta(\alpha_j)$).

In types $A_n$ and $D_n$, the fundamental weights $\eta(\alpha_i),\eta(\alpha_j)$ are minuscule, hence $|\langle \eta(\alpha_i), \alpha^{\vee}\rangle|=|\eta(\alpha_j), \sigma \alpha^{\vee}\rangle|=1$ (by the assumptions above both pairings are nonzero).

In types $B_n$ and $C_n$ ($n \geq 2$),
the pairing $\langle \eta(\alpha_i), \alpha^{\vee} \rangle$ has absolute value 1 if $\alpha^{\vee}$ is short and 2 if $\alpha^{\vee}$ is long (again because the pairing was assumed nonzero).  We conclude by observing that  the property of being long (resp. short) is preserved under $W \rtimes \galq$.

\end{proof}
\subsection{Invariance of the Hodge ray}
\begin{proof}[Proof of \textnormal{\Cor~\ref{cor-intro-hodge-line-invariant}}:]
By the proof of \Th~\ref{th-intro-omega-quasi-constant}, specifically \S\ref{sec-fund-weight-lemma}, one has $m \in \ZZ$ such that $\eta_{\omega,i}=m\eta(\alpha_i)$ when $\tilde L_i \neq \tilde G_i$ and $\eta_{\omega, i}=0$ when $\tilde L_i=\tilde G_i$. It remains to show that $m<0$. For this purpose, we use the dictionary between ample line bundles on a flag variety and dominant regular weights (\cf \cite[II.4.4]{jantzen-representations} and the ensuing remarks). 

Let $\PP$ be the parabolic subgroup of $\GG_E$ with Levi $\LL$ which stabilizes the Hodge filtration of $\ad \circ h$. By our conventions, given $\alpha \in \Phi \setminus \Phi(\LL, \TT)$, the root group $U_{\alpha}$ is contained in $\PP_{\qbar}$ if and only if $\alpha$ is negative. Let $I \subset \Delta$ be the type of $\PP$.

Write $\Pcal$ for the flag variety $\GG_{\qbar}/\PP_{\qbar}$ and $\Pcal_g$ in the Siegel case.
Over $\CC$, the projective variety $\Pcal_\CC$ is known as the compact dual of $\XX$. Given $\lambda \in X^*(\LL)$, the associated line bundle $\Lscr(\lambda)$ on $\Pcal$ is ample if and only if $\langle \lambda, \alpha^{\vee} \rangle>0$ for all $\alpha \in \Delta\setminus I$ (\loccitn).

The embedding~\eqref{eq-symplectic-embedding} induces an embedding of compact duals
$\Pcal \hookrightarrow \Pcal_g$. A first application of the above dictionary gives that, in the Siegel case, the Hodge line bundle $\omega_g$  is anti-ample on $\Pcal_g$. Since the pullback of an ample line bundle along a finite map is ample, the line bundle $\omega(\psi)$ is anti-ample on $\Pcal$. Thus a second application of the dictionary gives $\langle \eta_{\omega}, \alpha^{\vee} \rangle <0$ for all $\alpha \in \Delta \setminus I$. It follows that $m<0$ as desired.

\end{proof}
\subsection{A counter-example when $\GG^{\ad}$ is not $\QQ$-simple} \label{sec-ex-G-not-simple} 
In \Rmk~\ref{rmk-hodge-line}, it was claimed that the assumption "$\GG^{\ad}$ is $\QQ$-simple" is necessary for the invariance of the ray generated by the Hodge line bundle (\Cor~\ref{cor-intro-hodge-line-invariant}). Intuitively, it seems natural to expect that the factors of the Hodge line bundle $\omega(\psi)$ corresponding to different $\QQ$-factors of $\GG$ can vary independently of one another as one varies the symplectic embedding $\psi$ \eqref{eq-symplectic-embedding}. Some extra care is required because of the technical restrictions that a symplectic embedding of $\gx$ imposes on the center of $\GG$. So we give an example to show that the above intuition is in fact correct.

Let $g \geq 1$ be an integer. 
Let $\HH_g$ be the connected, reductive, split $\QQ$-group defined as the subgroup of the $g$-fold product $GL(2)^g$  where the $g$ components have the same determinant. Then the center of $\HH_g$ is one-dimensional and one has an isomorphism of $\QQ$-groups $\HH_g^{\ad} \cong PGL(2)^g$. 

Let $a,b$ be positive integers. For every $\QQ$-algebra $R$, mapping $(A,B) \in \HH_2(R)$ to the vector $(A,\ldots ,B,\ldots) \in \HH_{a+b}(R)$ with $a$ components equal to $A$ followed by $b$ components equal to $B$ gives an embedding $f_{a,b}:\HH_2 \hookrightarrow \HH_{a+b}$.

The group $\HH_g$ is isomorphic to the maximal rank subgroup of $GSp(2g)$ given by the sub-root system of all long roots in $GSp(2g)$ (in the notation of~\eqref{eq_roots_type_B-C}, this sub-root system is $O_2$). This gives an embedding of split $\QQ$-groups $k_g:\HH_g \hookrightarrow GSp(2g)$. Equivalently, if one defines $GSp(2g)$ using the alternating form on $\QQ^{2g}$ given by $J_g:=\left( \begin{array}{cc} 0 & -I_g \\
I_g & 0
\end{array}\right)$, then the image may also be described as $g$ "embedded squares", where the $i$th $2 \times 2$ square occupies the $(i,i),(i, g+i),(g+i,i),(g+i,g+i)$ entries of a $2g \times 2g$ matrix in $GSp(2g, R)$ (and all the other entries are $0$).   

As usual write $z=x+iy$ for $z \in \SS(\RR)$ (\S\ref{sec-set-up-hodge-line}). Define $h_g:\SS \to GSp(2g)_{\RR}$ corresponding to the standard formula $$h_g(z) \mapsto \left(\begin{array}{cc}
xI_g & -yI_g \\ yI_g & xI_g
\end{array}\right).$$ Recall that the $GSp(2g, \RR)$-conjugacy class of $h_g$ is $\XX_g$ pertaining to the Siegel datum $(GSp(2g), \XX_g)$.

Define $j_{g}: \SS \to (\HH_g)_{\RR}$  by $j_{g}(z)=(h_1(z), \ldots ,h_1(z))$, \ie $j_g$ is $h_1$ followed by the diagonal embedding.  Let $\YY_g$ be the $\HH_g(\RR)$-conjugacy class of $j_g$. Note that $\YY_g$ is different from the product $\XX_1^g$.
The pair $(\HH_g, \YY_g)$ is a Shimura datum.

Fix $g=a+b$. One has $j_g=f_{a,b} \circ j_2$ and $h_g=k_g \circ j_g$. Thus:
\begin{example} 
\label{ex-not-simple} The embeddings $f_{a,b}:\HH_2 \hookrightarrow \HH_{a+b}$ and $k_g: \HH_g \hookrightarrow GSp(2g)$ defined above satisfy:
\begin{enumerate}
\item Both $f_{a,b}$ and $k_g$ induce morphisms of Shimura data.
\item Both $k_g$ and $k_g \circ f_{a,b}$ are symplectic embeddings; in particular $(\HH_g, \YY_g)$ is of Hodge type.

\end{enumerate}
\end{example}
It remains to understand the Hodge characters associated to the embeddings in Example~\ref{ex-not-simple}.
With respect to the alternating form $J$ above, a $\QQ$-split
maximal torus $\TT_g$ in $GSp(2g)$ is given by the diagonal matrices of the form $\diag(t_1c, \ldots, t_gc, t_1^{-1}c, \ldots t_g^{-1}c)$. The Hodge character $\eta_g \in X^*(\TT_g)$ of $(GSp(2g), \XX_g)$ is then given by $(t_1 \cdots t_g)^{-1}c^{-g}$ and the pullback $s^*\eta_g$ to $\tilde \TT_g$ in $Sp(2g)$ is $(t_1\cdots t_g)^{-1})$. 

A compatible choice of $\QQ$-split maximal torus in $\HH_g$ is given by $g$-tuples of the form $$\left(c\left(\begin{array}{cc} s_1 & \\ &s_1^{-1}
\end{array}\right),\ldots,c\left(\begin{array}{cc} s_g & \\ &s_g^{-1}
\end{array}\right) \right).$$ Thus one finds:
\begin{example} The Hodge characters of \textnormal{Example~\ref{ex-not-simple}} satisfy the following: 
\begin{enumerate}
\item The Hodge character $k_g^*\eta_g$ is given by $(s_1 \cdots s_g)^{-1}c^{-g}$.
\item The Hodge character $(k_g \circ f_{a,b})^*\eta_g$ is given by $s_1^{-a}s_2^{-b}c^{-g}$.
\item The pullback of $(k_g \circ f_{a,b})^*\eta_g$ to the corresponding maximal torus of $\tilde \HH_2=SL(2) \times SL(2)$ is $s_1^{-a}s_2^{-b}$.
\end{enumerate}   
In particular, \textnormal{\Cor~\ref{cor-intro-hodge-line-invariant}} fails miserably for the symplectic embedding $(k_g \circ f_{a,b}):(\HH_2, \YY_2) \hookrightarrow (GSp(2g), \XX_{g})$.
\end{example}
The example above for $\HH_2$ is easily generalized in several ways, in particular to all $\HH_g$.
\section{Further applications, motivation and open problems} \label{sec-further-motivation}
\subsection{Uniform principal purity for quasi-constant cocharacters}
\label{sec-purity}
As a further application of quasi-constant (co)characters, we combine the duality construction for quasi-constant cocharacters (\Prop~\ref{prop-duality}) with our previous results on Hasse generators in \cite{Goldring-Koskivirta-Strata-Hasse,Goldring-Koskivirta-zip-flags} to deduce \Th~\ref{th-hasse-quasi-constant}. 

For the convenience of the reader, we recall the main result \cite[\Th~3.2.3]{Goldring-Koskivirta-Strata-Hasse} and the notions `orbitally $p$-close', `$L$-ample' and `$(p, L)$-admissible' which are used in its formulation  (\loccitn, \Defs~N.5.1,~N.5.3). Let $G$ be a connected, reductive $\fp$-group, $\mu \in X_*(T)$ a cocharacter and set $L:=\cent(\mu)$. Define $\Delta_L \subset \Delta$ by $\Delta_L:=\Delta \cap \Phi(L,T)$. A character $\chi \in X^*(T)$ is \underline{orbitally $p$-close} if, for all 
$\alpha \in \Phi$ satisfying $\langle \chi, \alpha^{\vee} \rangle \neq 0$ and all  $\sigma \in W \rtimes \galfp$, one has
\begin{equation} 
\label{eq-orb-p-close}
\left| \frac{ \langle \chi, \sigma\alpha^{\vee} \rangle}{\langle \chi, \alpha^{\vee} \rangle} \right| \leq p-1. \end{equation}
Further, $\chi$ is \underline{$L$-ample} if $\langle \chi , \alpha^{\vee} \rangle<0$ for all $\alpha \in \Delta \setminus \Delta_L$. Finally, $\chi$ is \underline{$(p, L)$ admissible} if it is both orbitally $p$-close and $L$-ample. Then \Th~3.2.3 of \loccit states that every $(p, L)$-admissible $\chi \in X^*(T)$ is a Hasse generator of $\GZip^\mu$ (\S\ref{sec-intro-hasse}).

The orbitally $p$-close condition \ref{eq-orb-p-close} is a natural weakening -- depending on $p$ -- of the quasi-constant one~\eqref{def quasi constant}; in particular a quasi-constant character is orbitally $p$-close for all primes $p$. As for `quasi-constant',  the condition `orbitally $p$-close' naturally extends to an element of $X^*(T)_{\QQ}$, even if it is not a character. The character $\chi$ is $L$-ample if and only if the associated line bundle $\Lscr(\chi)$ is anti-ample on the flag variety $G/P$, where $P$ is the parabolic of type $\Delta_L$ containing $L$ (\S\ref{sec-ex-G-not-simple}).

\begin{proof}[Proof of \textnormal{\Th~\ref{th-hasse-quasi-constant}}:] By assumption $\mu$ is quasi-constant. Without loss of generality, we may assume $\mu$ is $\Delta$-dominant. Let $\langle \mu \rangle$ be the quasi-constant ray spanned by $\mu$. By \Prop~\ref{prop-duality}, the dual ray
$\langle \mu \rangle^{\vee} \subset X^*(T)_{\QQ}$ afforded by Construction~\ref{const-duality} is quasi-constant.  
Let $\mu^*$ be a nontrivial element of $\langle \mu \rangle^{\vee} \cap X^*(T)$. Then $\mu^*$ is quasi-constant. This proves~\ref{item-levi-admits-quasi-constant}. 

By construction, $\mu^*$ is $\Delta$-dominant. Thus $-\mu^*$ is $\Delta$-anti-dominant. By ~\eqref{eq-dual-quasi-constant-line},  $-\mu^*$ is $L$-ample. Therefore $-\mu^*$ is $(p, L)$-admissible. So part~\ref{item-quasi-constant-implies-hasse-inv-generator} follows from \cite[\Th~3.2.3]{Goldring-Koskivirta-Strata-Hasse}.
\end{proof}
\subsection{`Quasi-constant' as unification of `minusucule' and `cominuscule'}
\label{sec-motivation-unification-co-min}

In view of \Th~\ref{th-quasi-constant-cominuscule}, when $G$ is simple and $k$ is algebraically closed, the property `quasi-constant' captures the union of the two properties `minuscule' and `cominuscule', up to scalar multiples.

The equivalent definitions of `cominuscule' which appear in the literature (\cf
\S\ref{sec-cominuscule}) have several drawbacks.
First, they are only valid for semisimple $G$ (\footnote{An artificial extension to reductive $G$ can be given, for instance, by declaring that $\chi$ is cominuscule if its restriction to a maximal torus of the derived subgroup of $G$ is cominuscule.}).
This goes against the philosophy of Deligne, Serre, Langlands and others which highlights the importance (and necessity) of considering {\it all} connected reductive groups.
Second, even for semisimple $G$, the definition of cominuscule requires choosing a basis $\Delta \subset \Phi$ of simple roots.
Third, the definition of cominuscule makes reference to `minuscule' and presupposes that the relationship between `minuscule' and `fundamental weights' has already been understood in the semisimple case.

By contrast, both the definitions of `minuscule' (\cf \S\ref{sec-minuscule}) and  `quasi-constant' (\Def~\ref{def quasi constant})
 have none of these issues: They apply uniformly to general $G$, require no choice of basis and do not presuppose anything beyond the root datum of $(G,T)$.

For these reasons, we suggest that a conceptual implication of \Th~\ref{th-quasi-constant-cominuscule} may be that, among `cominuscule' and `quasi-constant', the latter is the more natural notion. The validity of our suggestion should be tested by applying the above two notions in various different contexts.

\subsection{`Quasi-constant' as a test case in our program} \label{sec-test-case}
Recall that, as mentioned in \S\ref{sec-intro}, our general program aims to connect (A) Automorphic Algebraicity, (B) $\GZip$-Geometricity, and (C) Griffiths-Schmid Algebraicity. The basic objects in (B) and (C) -- stacks $\GZip^\mu$ and Griffiths-Schmid manifolds -- are both essentially associated to data $(G, [\mu])$, where $G$ is a connected, reductive $k$-group and $[\mu]$ is the conjugacy class of a cocharacter $\mu \in X_*(G)$. In the case of $\GZip^\mu$, $k=\fp$, while for Griffiths-Schmid manifolds $k=\QQ$. 
As we briefly recall below, much more is known about both (B) and (C) when the cocharacter $\mu$ is minuscule, thanks to the theory of Shimura varieties\footnote{Many of the more sophisticated results concerning Shimura varieties require the more stringent hypothesis that $\mu$ is of Hodge or abelian type.}. 

It is therefore natural to seek generalizations of the minuscule condition on which to test questions and conjectures regarding $\GZip^\mu$, Griffiths-Schmid manifolds and the connections between the two.  We propose the quasi-constant condition as such a generalization. Below, we single out three questions concerning (B)-(C)  about which a considerable amount is known in the minuscule case, but which are wide-open beyond that.

In addition to their intrinsic interest and contribution to our program, progress on these questions is likely to have significant applications to the Langlands correspondence between automorphic representations and Galois representations. The link between the Langlands correspondence and Griffiths-Schmid algebraicity was studied extensively in Carayol's program (see \cite{Carayol-LDS-1,Carayol-LDS-2,Carayol-LDS-3,Carayol-laumon-volume} and \cite{Green-Griffiths-Kerr-CBMS-Texas}). In the context of Hodge-type Shimura varieties, applications of the link with $G$-Zips to the Langlands correspondence were studied in \cite{Goldring-Koskivirta-Strata-Hasse}, where in many cases Galois representations were associated to automorphic representations with non-degenerate limit of discrete series archimedean component, and pseudo-representations were associated to spaces of coherent cohomology modulo a prime power.

\subsubsection{Griffiths-Schmid manifolds} \label{sec-GS-mflds}
The complex manifolds that bear their name were introduced by Griffiths-Schmid almost half-a-century ago in 1969, \cite{Griffiths-Schmid-homogeneous-complex-manifolds}.
However, their study underwent several decades of relative hibernation, until it was revived by Carayol in a series of papers initiated in the late 1990's and later also in a series of works by Griffiths and his school (\cf \cite{Green-Griffiths-Kerr-Mumford-Tate-Domains-book,Green-Griffiths-Kerr-CBMS-Texas,KerrSp4,Griffiths-Robles-Toledo-not-algebraic}).
The main cause for the dormant period was probably that, since their introduction, it was widely believed that -- in a precise sense recalled below --  `most' Griffiths-Schmid manifolds are not algebraic. This belief was recently confirmed by Griffiths-Robles-Toledo \cite{Griffiths-Robles-Toledo-not-algebraic}.

Suppose $\GG$ is a connected, reductive $\QQ$-group and $\XX$ is a $\GG(\RR)$-conjugacy class of a morphism of $\RR$-groups $h:\SS \to G_{\RR}$ satisfying Deligne's axioms for a Shimura variety (2.1.1.2) and (2.1.1.3) of \cite{Deligne-Shimura-varieties}, but not necessarily satisfying axiom (2.1.1.1) of \loccit
That is, assume that $\ad h(i)$ is a Cartan involution of $\GG_{\RR}^{\ad}$ and that no $\QQ$-simple factor of $\GG^{\ad}$ has compact real points; contrary to the case of a Shimura variety we do not assume that the Hodge structure $\ad \circ h$ on $\Lie(\GG)_{\CC}$ is of type
$\{(1,-1), (0,0), (-1,1)\}$.

By the work of Griffiths-Schmid \cite{Griffiths-Schmid-homogeneous-complex-manifolds}, reinterpreted in the language of \cite{Deligne-Shimura-varieties} (see also \cite{Carayol-laumon-volume} and \cite{Patrikis-MT-groups} for the translation), one has a projective system of Griffiths-Schmid (complex) manifolds
\begin{equation}
\label{eq-gs-tower}
(\gsgxk)_{\Kcal \subset \gofaf},
\end{equation} indexed by (neat) open compact subgroups $\Kcal$ of $\gofaf$. The system~\eqref{eq-gs-tower} admits an action of $\gofaf$ in the sense of \cite[2.1.4]{Deligne-Shimura-varieties}. In \cite{Griffiths-Robles-Toledo-not-algebraic}, it is shown that $\gsgxk$ is not algebraic -- in the sense that $\gsgx$ is not the analytification $X^{\an}$ of a $\CC$-scheme $X$ -- unless the following condition, termed the `classical case' \footnote{This condition is sometimes also called the `semi-classical' case, to distinguish it from the case of an actual Shimura datum.} in \loccitn, is satisfied:

\begin{enumerate}
\item[(Cl)]
\label{item-classical-case} There exists a Shimura datum $(\GG, \XX')$ (with the same underlying group $\GG$) such that the natural smooth map $\XX \to \XX'$ is holomorphic.
\end{enumerate}
One way to understand the map $\XX \to \XX'$ is as follows:
The centralizer of $h' \in \XX'$ is a maximal connected, compact modulo center subgroup of $\gofr$. The centralizer of any $h \in \XX$ is also connected and compact modulo center, but possibly not maximal.
Thus one can choose $h \in \XX$ so that $\stab(h) \subset \stab(h')$.
The map $\XX \to \XX'$ is then simply the projection $\gofr/\stab(h) \to \gofr/\stab(h')$. 

Condition (Cl) is equivalent to $\gsgxk$ being the analytification of a (partial) flag space associated to the Shimura variety  $\shgxk$ as defined in \cite{Goldring-Koskivirta-zip-flags} (an important special case is already discussed in \cite[\S9.1]{Goldring-Koskivirta-Strata-Hasse}).
Briefly, the (partial) flag spaces of a Shimura variety are algebraic fibrations over the Shimura variety with (partial) flag variety fibers.
For Shimura varieties of Hodge type, the integral models of Kisin \cite{Kisin-Hodge-Type-Shimura} and Vasiu \cite{Vasiu-Preabelian-integral-canonical-models} can be used to construct integral models of the associated flag spaces, see \cite{Goldring-Koskivirta-zip-flags}.

Following Deligne, one sets $\mu=\mu_h=(h\otimes \CC)(z,1)$ for $h \in \XX$ to obtain a cocharacter $\mu \in X_*(G)$ and thus a pair $(G, [\mu])$. Conversely, the pair $(\GG, [\mu])$ almost determines a pair $(\GG', \XX)$; there are subtleties having to do with the center and the real form $\GG'_\RR$ determined by $\mu$ may be different than $\GG_\RR$, see \cite[1.2.4]{Deligne-Shimura-varieties} for details.
\subsubsection{Algebraicity of Griffiths-Schmid manifolds} \label{sec-alg-GS}
Notwithstanding the negative result of \cite{Griffiths-Robles-Toledo-not-algebraic}, there are several poignant reasons to believe that there is a hidden algebraicity underlying {\it all} Griffiths-Schmid manifolds. Some such reasons which arise from Hodge theory are discussed in the aforementioned references of Griffiths and his collaborators. We shall now briefly mention the reason underlying Carayol's program.

Carayol observed cases where automorphic representations $\pi$ with degenerate limit of discrete series  archimedean component contribute to the cohomology of non-classical Griffiths-Schmid manifolds.
More precisely, this means that one has a $\gofaf$-equivariant embedding of the finite part $\pi_f$ into $\varinjlim_{\Kcal} H^i(\gsgxk, \llambda) $, for some $i$ and some automorphic line bundle $\llambda$. When $\gsgxk$ is also compact, Carayol shows that in fact every cohomology class in $H^i(\gsgxk, \llambda)$ is represented by automorphic forms.
The relationship between automorphic representations and the cohomology of Griffiths-Schmid manifolds observed by Carayol in particular examples (see also Kerr \cite{KerrSp4} and Charbord \cite{Charbord-thesis} for further examples) are expected to hold for all Griffiths-Schmid manifolds.

At this point, an intuition for some form of algebraicity for Griffiths-Schmid manifolds comes from the Langlands program.
The automorphic representations $\pi$ which contribute to the cohomology of Griffiths-Schmid manifolds are all necessarily $C$-algebraic in the sense of Buzzard-Gee \cite{Buzzard-Gee-conjectures}.
The Langlands program conjectures that $C$-algebraic automorphic representations $\pi$ should enjoy a wide variety of algebraicity properties.
For example, the Hecke eigenvalues (Satake parameters) of $\pi$ should be algebraic numbers and there should be a compatible system of Galois representations (ultimately a motive) associated to  $\pi$.
See \loccit for some precise conjectures along these lines.

Combining the remarks above about the link between cohomology and automorphic representations on the one hand and the Langlands program on the other, one is led to suspect, as Carayol did, that at least the coherent cohomology of automorphic line bundles on Griffiths-Schmid manifolds is deeply algebraic; for example that it should admit a $\qbar$-structure. Since properties of the cohomology of a space $X$ should reflect those of $X$ itself, we are led to ask:
\begin{question}
Is there a generalized notion of algebraicity which is satisfied by all Griffiths-Schmid manifolds?
\label{q-further-motivation-alg-GS}\end{question}

\subsubsection{Geometrization of \texorpdfstring{$\gzipmu$}{G-Zip}} \label{sec-geometrization-gzipmu}

The underlying topological space of $\GZip^\mu$ is a finite set of points. Thus it seems that $\GZip^\mu$ lacks some global geometric richness.  One way to apply the theory of $\GZip^\mu$ to schemes $X$ is to study morphisms $X \to \GZip^\mu$. This raises two problems: The first is to exhibit interesting examples of $X \to \GZip^\mu$. The second was singled out as Question B in the introduction to \cite{Goldring-Koskivirta-Diamond-I}: To what extent is the geometry of $X$ controlled by $\GZip^\mu$ and properties of a morphism $X \to \GZip^\mu$? 

Regarding the first problem, Shimura varieties of Hodge type furnish important examples of morphisms $X \to \GZip^\mu$. More precisely, suppose $\gx$ is a Shimura datum of Hodge type, $p$ is a prime at which $\GG$ is unramified and $(G, \mu)$ arises from $(\GG, [\mu])$ by reduction mod $p$. If $\Kcal \subset \gofaf$ is hyperspecial at $p$, then a theorem of Zhang asserts that there is a smooth morphism from the special fiber of the Kisin-Vasiu $p$-integral model of $\shgx_{\Kcal}$ to $\gzipmu$, \cite{ZhangEOHodge}. 

Concerning the second problem, the works \cite{Goldring-Koskivirta-Strata-Hasse,Goldring-Koskivirta-zip-flags,Goldring-Koskivirta-Diamond-I,Brunebarbe-Goldring-Koskivirta-Stroh-ampleness} give various positive examples of geometric properties that are to a large extent controlled by properties of a morphism $X \to \GZip^\mu$. These include the existence of global sections and positivity of certain vector bundles on $X$, as well as the affineness of Ekedahl-Oort strata.

Our second question is then:
\begin{question} Is there a generalization of Zhang's morphism \begin{equation}
    X\to \gzipmu
\end{equation}  
for more general pairs $(G, \mu)$?
\label{q-scheme-to-GZip}
\end{question}

In analogy with the case of Hodge-type Shimura varieties, a more optimistic and more precise question would be to ask for an entire system $(X_\Kcal)_\Kcal \subset \gofaf$ mapping to $\GZip^\mu$, where $\Kcal$ runs over subgroups which are (a fixed) hyperspecial at $p$. 

\subsubsection{Extending the link between Griffiths-Schmid and \texorpdfstring{$\gzipmu$}{G-Zip}
} \label{sec-extending-link-GS-gzipmu}

Since Shimura varieties of Hodge-type are a very special case of Griffiths-Schmid manifolds, Zhang's theorem provides a direct link between a subclass of Griffiths-Schmid manifolds and subclass of the $\GZip^\mu$.  

\begin{question} Does the link between the stacks $\gzipmu$ and the Griffiths-Schmid manifolds $\gsgx_{\Kcal}$ extend beyond the case that $\mu$ is of Hodge type and even beyond the case that $\mu$ is minuscule? \label{q-extending-link-GS-gzipmu}
\end{question}
In a small number of cases, this is achieved by the theory of Zip flags, see \cite[\S\S2,9.1]{Goldring-Koskivirta-Strata-Hasse} and especially \cite{Goldring-Koskivirta-zip-flags}, where the scheme $X$ is given by the partial flag spaces of a Shimura variety.

\subsubsection{The case of quasi-constant cocharacters} \label{sec-further-motivation-test-case} Returning to the fundamental notion of this paper -- the quasi-constant condition -- we conclude with:
\begin{question}
\label{q-further-motivation-test-case} Is there a special approach or simpler answer to \textnormal{Questions~\ref{q-further-motivation-alg-GS}},\textnormal{~\ref{q-scheme-to-GZip}} and\textnormal{~\ref{q-extending-link-GS-gzipmu}}
when $\mu$ is quasi-constant?

\end{question}

In this paper, we gave an example of a different question about cocharacter data, namely Question~\ref{q-principally-pure} on the uniform principal purity of $\gzipmu$, where we were able to provide a positive answer to the analogue of Question~\ref{q-further-motivation-test-case}.

\AtBeginEnvironment{equation}{\setcounter{equation}{\value{subsection}}}{}{}
\AtEndEnvironment{equation}{\stepcounter{subsection}}{}{}
\appendix
\numberwithin{equation}{subsection}
\section{Explicit bounds for uniform principal purity} 

\label{app-bounds}

Return to the setting of \S\ref{sec-intro-hasse} and \S\ref{sec-purity}: $G$ is a connected, reductive $\fp$-group, $\mu \in X_*(G)$ and $L=\cent(\mu)$. For simplicity, we shall assume throughout the appendix that $G^{\ad}$ is absolutely simple, \ie that $G^{\ad}_{\fpbar}$ is simple. Fundamental weights will refer to those of $G^{\ad}$.

It is natural to seek an explicit bound $C(\Delta, \Delta_L)$, depending only on the root system of $G$ and the type of $L$, such that $\GZip^{\mu}$ is uniformly principally pure provided that $p > C(\Delta, \Delta_L)$. When $\mu$ is quasi-constant, it was shown in \Th~\ref{th-hasse-quasi-constant} that uniform principal purity holds for all $p$, \ie one may take $C(\Delta, \Delta_L)=1$. Below we record an explicit upper bound for $C(\Delta, \Delta_L)$, for every irreducible $\Delta$ and every Levi type $\Delta_L$.    

To this end, we use an elementary application of our result on group-theoretical Hasse invariants \cite[\Th~3.2.3]{Goldring-Koskivirta-Strata-Hasse} (which was recalled in \S\ref{sec-purity}). 
Recall also the notation for fundamental weights in \S\ref{sec-fund-weights}. Given $S \subset \Delta$, write $\eta(S)=\sum_{\alpha \in S} \eta(\alpha).$  

\begin{lemmaapp}
\label{lem-app-fund-wts-pure}
Assume $\eta(\Delta\setminus \Delta_L)$ is orbitally $p$-close \textnormal{(\S\ref{sec-purity})}. If $\chi \in X^*(T^{\ad})$ is a strictly negative multiple of $\eta(\Delta\setminus \Delta_L)$, then $\pr^*\chi$ is a Hasse generator for $\GZip^\mu$ \textnormal{(\S\ref{sec-intro-hasse})}. Consequently, 
$\GZip^\mu$ is uniformly principally pure.  
\end{lemmaapp}
\begin{proof} Suppose $\chi$ is  a character of $T^{\ad}$ which is a negative multiple of $\eta(\Delta \setminus \Delta_L)$.
By definition, the condition `orbitally $p$-close' is closed under non-zero scalar multiples: 
For every $c \in \QQ^{\times}$,  $\xi \in X^*(T)_{\QQ}$ is orbitally $p$-close if and only  $c\xi$ is. 
Further $\xi$ is orbitally $p$-close if and only if $\pr^*(\chi)$ is. 
Thus both $\chi$ and $\pr^*\chi$ are orbitally $p$-close by assumption. 
Since $\chi$ is a negative multiple of $\eta(\Delta\setminus \Delta_L)$, $\pr^*\chi$ is $L$-ample (\S\ref{sec-purity}). By \opcitn, $\chi$ is a Hasse generator of $\GZip^\mu$.  
\end{proof}
One can simplify the computation of whether $\eta(\Delta\setminus \Delta_L)$ is orbitally $p$-close using a reduction to the $\Delta$-dominant elements of $\Phi^{\vee}$. If $\Dfr$ is simply-laced (\S\ref{sec-dynkin}), then $\Phi^{\vee}$ contains a unique $\Delta$-dominant element, namely the highest coroot $ {}^h\alpha^{\vee}$ (\S\ref{sec-root-mult}). Otherwise $\Dfr$ is multi-laced and then $\Phi^{\vee}$  contains precisely two $\Delta$-dominant elements: the highest coroot ${}^h\alpha^{\vee}$ and the coroot $({}^h\alpha)^{\vee}$ corresponding to the highest root.
\begin{lemmaapp}\ 
\label{lem-orb-p-close-highest-root}
\begin{enumerate}
\item 
\label{item-lem-app-simply}
Assume that either $\Dfr$ is simply-laced or $\Delta \setminus \Delta_L$ contains a short root. Then  $\eta(\Delta\setminus \Delta_L)$   
is orbitally $p$-close if and only if $\langle \eta(\Delta\setminus \Delta_L),{}^h\alpha^{\vee} \rangle  \leq p-1 $.
\item 
\label{item-lem-app-multi}
Assume $\Dfr$ is multi-laced and every root in $\Delta\setminus \Delta_L$ is long. Then $\eta(\Delta\setminus \Delta_L)$   
is orbitally $p$-close if and only if  $\langle \eta(\Delta\setminus \Delta_L),({}^h\alpha)^{\vee} \rangle  \leq p-1 $.

\end{enumerate}  
\end{lemmaapp}
\begin{proof} Since $G$ is assumed absolutely simple, the orbits of $W \rtimes \galfp$ on $\Phi$ and $\Phi^{\vee}$ agree with those of $W$. Recall that for every positive coroot $\alpha^{\vee}$, the difference ${}^h\alpha^{\vee}-\alpha^{\vee}$ is a nonnegative $\ZZ$-linear combination of simple coroots. 
Since the fundamental weights are $\Delta$-dominant, for $\alpha$ ranging over $\Phi$, the value of $\langle \eta(\Delta\setminus \Delta_L), \alpha^{\vee} \rangle$ is maximal when $\alpha^{\vee}={}^h\alpha^{\vee}$. 
Moreover, when $\Dfr$ is multi-laced, the highest coroot is always long.  

Under either assumption in~\ref{item-lem-app-simply}, $\Delta^{\vee}\setminus \Delta^{\vee}_L$ contains a long coroot $\gamma^{\vee}$. 
Thus $\gamma^{\vee}$ is in the same Weyl group orbit as ${}^h\alpha^{\vee}$. Since $\langle \eta(\Delta\setminus \Delta_L), \gamma^{\vee} \rangle=1$, as $\alpha$ ranges over those roots in $\Phi$ satisfying $\langle \eta(\Delta \setminus \Delta_L), \alpha^{\vee} \rangle \neq 0$ and  $\sigma$ ranges over $W \rtimes \galfp$,
 the maximal value of 
\begin{equation} \tag{A.3}
\label{eq-lem-app}
\left| \frac{ \langle \eta(\Delta\setminus \Delta_L), \sigma\alpha^{\vee} \rangle}{\langle \eta(\Delta\setminus \Delta_L), \alpha^{\vee} \rangle} \right| \end{equation} is $\langle \eta(\Delta\setminus \Delta_L), {}^h\alpha^{\vee} \rangle$. This proves~\ref{item-lem-app-simply}.

Now consider~\ref{item-lem-app-multi}. The hypothesis ensures that every coroot in $\Delta^{\vee} \setminus \Delta_L^{\vee}$ is in the same $W$-orbit as $({}^h\alpha)^{\vee}$. The same reasoning as in ~\ref{item-lem-app-simply} shows that the maximal value attained in~\ref{eq-lem-app}
is $\langle \eta(\Delta\setminus \Delta_L), ({}^h\alpha)^{\vee} \rangle$, at least when $\alpha^{\vee}$ is in the short $W$-orbit of coroots. It remains to check that no larger value occurs in~\eqref{eq-lem-app} when $\alpha^{\vee}$ is in the long orbit of coroots. This can be checked by hand, using the identifications recalled in the proofs of Lemmas~\ref{lem-G2},~\ref{lem-F4} and~\ref{lem-BC}. 
\end{proof}
\begin{rmkapp} The pairing $\langle \eta(\Delta\setminus \Delta_L),{}^h\alpha^{\vee} \rangle$ equals the sum of the coroot multiplicities $\sum_{\alpha \in \Delta \setminus \Delta_L} m^{\vee}(\alpha)$ (\S\ref{sec-root-mult}). These multiplicities are given in both \cite[\Chap~VI, Planches I-IX]{bourbaki-lie-4-6} and \cite[Appendix C.1-C.2]{Knapp-beyond-intro-book}.
\end{rmkapp}
For every Dynkin diagram $\Dfr$ of a reduced and irreducible root system, $\alpha \in \Delta$ and $\Delta_L=\Delta\setminus \{\alpha\}$,  the table below gives the bound $C(\Delta,\Delta_L)$ for uniform principal purity  gotten by combining Lemmas~\ref{lem-app-fund-wts-pure} and~\ref{lem-orb-p-close-highest-root}.  The names of the simple roots refer to \cite[\Chap~VI, Planches I-IX]{bourbaki-lie-4-6}; in the multi-laced case they agree with the notation of Lemmas~\ref{lem-G2},~\ref{lem-F4} and~\ref{lem-BC}.
\begin{table}[h] \label{table-bounds}
        \caption{Bounds for uniform principal purity}

{\renewcommand{\arraystretch}{1.25}      
\begin{tabular}{|c|c|c|}
\hline
 Type of $\Dfr$ & Simple root $\alpha$ & $C(\Delta, \Delta_L)$ \\ \hline
 $A_n$ & $\alpha_i=e_i-e_{i+1}$, ($1 \leq i \leq n$) & $1$ \\ \hline
 \multirow{2}{*}{$B_n$} & $\alpha_1=e_1-e_2$, $\alpha_n=e_n$ & $1$ \\ 
\cline{2-3} & $\alpha_i=e_i-e_{i+1}$, ($2 \leq i \leq n-1$) & $2$ \\  \hline
\multirow{2}{*}{$C_n$} & $\alpha_1=e_1-e_2$, $\alpha_n=2e_n$ & $1$ \\ 
\cline{2-3} & $\alpha_i=e_i-e_{i+1}$, ($2 \leq i \leq n-1$) & $2$ \\  \hline
\multirow{2}{*}{$D_n$} & $\alpha_1=e_1-e_2$, $\alpha_{n-1}=e_{n-1}-e_n$, $\alpha_{n}=e_{n-1}+e_n$ & $1$ \\ 
\cline{2-3} & $\alpha_i=e_i-e_{i+1}$, ($2 \leq i \leq n-2$) & $2$ \\  \hline
$G_2$ & $\alpha_1,\alpha_2$  & $2$ \\ 
\hline
\multirow{2}{*}{$F_4$} & $\alpha_1, \alpha_4$ & $2$ \\ 
\cline{2-3} & $\alpha_2,\alpha_3$ & $3$ \\  \hline
\multirow{3}{*}{$E_6$} & $\alpha_1, \alpha_6$ & $1$ \\ 
\cline{2-3} & $\alpha_2,\alpha_3, \alpha_5$ & $2$ \\ 
\cline{2-3} & $\alpha_4$ & $3$ \\  \hline
\multirow{3}{*}{$E_7$} & $\alpha_7$ & $1$ \\ 
\cline{2-3} & $\alpha_1,\alpha_2, \alpha_6$ & $2$ \\ 
\cline{2-3} & $\alpha_3, \alpha_4, \alpha_5$ & $3$ \\  \hline
\multirow{3}{*}{$E_8$} & $\alpha_1,\alpha_8$ & $2$ \\ 
\cline{2-3} & $\alpha_2,\alpha_3, \alpha_6,\alpha_7$ & $3$ \\ 
\cline{2-3} & $\alpha_4, \alpha_5$ & $5$ \\  \hline
\end{tabular}}
\end{table}

An immediate corollary of the table is:
\begin{corapp} \ 
\begin{enumerate}
\item If $L$ is maximal and $p \geq 7$, then $\GZip^\mu$ is uniformly principally pure. 
\item If $L$ is maximal, $G$ is classical (type $A_n,B_n,C_n$ or $D_n$) or of type $G_2$ and $p \geq 3$, then $\GZip^\mu$ is uniformly principally pure. 
\end{enumerate}

\end{corapp}
\begin{rmkapp}
In all but the types $E_7, E_8$, the coroot multiplicities satisfy $m^{\vee}(\alpha) \in \{1,2,3\}$, so a bound $C(\Delta, \Delta_L)$ is obtained for a non-maximal $L$ by adding the bounds given in the table for the various $\alpha \in \Delta\setminus \Delta_L$. 
For $E_7,E_8$ there are $\alpha \in \Delta$ for which $m^{\vee}(\alpha)=m(\alpha)$ is composite (and >1): $m^{\vee}(\alpha_4)=4$ for $E_7$, while $m^{\vee}(\alpha_3)=m^{\vee}(\alpha_6)=4$ and $m^{\vee}(\alpha_4)=6$ for $E_8$. In these cases, one must add the true coroot multiplicity rather than the value in the table (the table gives the largest prime smaller than the multiplicity). For example, let $\Delta \setminus \Delta_L=\{\alpha_1,\alpha_3,\alpha_4\}$ in type $E_8$. Then the bound obtained is $C(\Delta, \Delta_L)=2+4+6=12$, while adding the values in the table gives $2+3+5=10$. In this example, we do not know if uniform principle purity holds for $p=11$.  
\end{rmkapp}

\bibliographystyle{plain}
\bibliography{biblio_overleaf}

\end{document}